\documentclass[11pt,reqno]{preprint}
\usepackage[full]{textcomp}
\usepackage[osf]{newtxtext}
\usepackage{comment}

\usepackage{amssymb}
\usepackage{mathtools}
\usepackage{hyperref}
\usepackage{breakurl}
\usepackage{mhenvs}
\usepackage{mhequ}
\usepackage{mhsymb}
\usepackage{booktabs}
\usepackage{tikz}
\usetikzlibrary{decorations.pathreplacing}
\usepackage{mathrsfs}
\usepackage{longtable}

\usepackage{microtype}
\usepackage{comment}
\usepackage{wasysym}
\usepackage{centernot}
\usepackage{enumitem}
\usepackage{bm}
\usepackage{stackrel}

\DeclareMathAlphabet{\mathbbm}{U}{bbm}{m}{n}

\DeclareFontFamily{U}{BOONDOX-calo}{\skewchar\font=45 }
\DeclareFontShape{U}{BOONDOX-calo}{m}{n}{
  <-> s*[1.05] BOONDOX-r-calo}{}
\DeclareFontShape{U}{BOONDOX-calo}{b}{n}{
  <-> s*[1.05] BOONDOX-b-calo}{}
\DeclareMathAlphabet{\mcb}{U}{BOONDOX-calo}{m}{n}
\SetMathAlphabet{\mcb}{bold}{U}{BOONDOX-calo}{b}{n}

\setlist{noitemsep,topsep=4pt}

\makeatletter
\def\DeclareSymbol#1#2#3{\expandafter\gdef\csname MH@symb@#1\endcsname{\tikz[baseline=#2,scale=0.15,draw=symbols,line join=round]{#3}}\expandafter\gdef\csname MH@symb@#1s\endcsname{\scalebox{0.7}{\tikz[baseline=#2,scale=0.15,draw=symbols,line join=round]{#3}}}}
\def\<#1>{\csname MH@symb@#1\endcsname}
\makeatother

\DeclareSymbol{0}{-2.4}{\node[dot] {};}
\DeclareSymbol{1}{0}{\draw[white] (-.5,0) -- (.5,0); \draw (0,0)  -- (0,1.5) node[sdot] {};}
\DeclareSymbol{11}{0}{\draw (-0.5,1.2) node[sdot] {} -- (0,0) -- (0.5,1.2) node[sdot] {};}
\DeclareSymbol{111}{0}{\draw (0,0) -- (0,1.2) node[sdot] {}; \draw (-.7,1) node[sdot] {} -- (0,0) -- (.7,1) node[sdot] {};}
\DeclareSymbol{131}{-3}{\draw (0,0) -- (0,-1) -- (1,0) node[sdot] {}; \draw (0,0) -- (0,-1) -- (-1,0) node[sdot] {}; \draw (0,0) -- (0,1.2) node[sdot] {}; \draw (-.7,1) node[sdot] {} -- (0,0) -- (.7,1) node[sdot] {};}
\DeclareSymbol{11}{0}{\draw (-0.5,1.2) node[sdot] {} -- (0,0) -- (0.5,1.2) node[sdot] {};}
\DeclareSymbol{12}{-3}{\draw (-.8,1) node[sdot] {} -- (0,0) -- (0,-1); \draw (1,0) node[sdot] {} -- (0,-1) -- (-1,0) node[sdot] {};}
\DeclareSymbol{30}{-3}{\draw (0,0) -- (0,-1); \draw (0,0) -- (0,1.2) node[sdot] {}; \draw (-.7,1) node[sdot] {} -- (0,0) -- (.7,1) node[sdot] {};}
\DeclareSymbol{10}{-3}{\draw (-.8,1) node[sdot] {} -- (0,0) -- (0,-1);}
\DeclareSymbol{22}{-3}{\draw (0,0.3) -- (0,-1) -- (1,0) node[sdot] {}; \draw (0,0.3) -- (0,-1) -- (-1,0) node[sdot] {};\draw (-.7,1) node[sdot] {} -- (0,0.3) -- (.7,1) node[sdot] {};}
\DeclareSymbol{211}{0}{\draw (-0.5,2.4) node[sdot] {} -- (-1,1.6);\draw (-1.5,2.4) node[sdot] {} -- (-0.5,0.8); \draw (0,1.6) node[sdot] {} -- (-0.5,0.8) -- (0,0) -- (0.5,0.8) node[sdot] {};}

\DeclareSymbol{Xi22}{0.5}{\draw (0,0) node[xi] {} -- (-1,1) node[not] {} -- (0,2) node[xi] {};}

\DeclareSymbol{Xi2}{-2}{\draw (0,-0.25) node[xi] {} -- (-1,1) node[xi] {};}
\DeclareSymbol{Xi3}{0}{\draw (0,0) node[xi] {} -- (-1,1) node[xi] {} -- (0,2) node[xi] {};}
\DeclareSymbol{Xi4}{2}{\draw (0,0) node[not] {} -- (-1,1) node[xi] {} -- (0,2) node[xi] {} -- (-1,3) node[xi] {};}
\DeclareSymbol{Xi2X}{-2}{\draw (0,-0.25) node[xi] {} -- (-1,1) node[xix] {};}
\DeclareSymbol{XXi2}{-2}{\draw (0,-0.25) node[xix] {} -- (-1,1) node[xi] {};}

\DeclareSymbol{IXi^2}{-1}{\draw (-1,1) node[xi] {} -- (0,0);
\draw[kernels2] (1,1) node[xi] {} -- (0,0) node[not] {};}

\DeclareSymbol{XiX}{-2.8}{\node[xibx] {};}
\DeclareSymbol{tauX}{-2.8}{ \draw[kernels2] (0,0) node[xibx] {};}
\DeclareSymbol{Xi}{-2.8}{\node[xib] {};}
\DeclareSymbol{XiY}{-2.8}{\node[xie] {};}
\DeclareSymbol{XiZ}{-2.8}{\node[xid] {};}
\DeclareSymbol{IXiX}{0}{\draw (0,-0.25) node[not] {} -- (0,1.5) node[xix] {};}

\newcommand{\cut}{\mathfrak{C}}

\newcommand{\mrd}{\mathop{}\!\mathrm{d}}
\newcommand{\ssep}{\mid}

\newcommand{\mcE}{\mathcal{E}}

\newcommand{\mcA}{\mathcal{A}}

\newcommand{\mcC}{\mathcal{C}}
\newcommand{\mcB}{\mathcal{B}}

\newcommand{\mcD}{\mathcal{D}}

\newcommand{\mcX}{\mathcal{X}}

\newcommand{\mbX}{\mathbf{X}}

\newcommand{\T}{\mathbf{T}}

\newcommand{\Bonds}{\mathbf{B}}
\newcommand{\Grid}{\mathbf{G}}

\def\${|\!|\!|}

\def\var#1{#1\textnormal{-var}}
\def\Hol#1{#1\textnormal{-H{\"o}l}}
\def\gr#1{#1\textnormal{-gr}}

\def\hol{\textnormal{hol}}

\def\scal#1{{\langle#1\rangle}}

\newcommand{\mfT}{\mathfrak{T}}

\newcommand{\mfA}{\mathfrak{A}}

\newcommand{\mfR}{\mathfrak{R}}

\newcommand{\mfp}{\mathfrak{p}}

\newcommand{\mfG}{\mathfrak{G}}
\newcommand{\mfg}{\mathfrak{g}}

\def\cC{\mathscr{C}}

\newcommand{\YM}{\mathrm{YM}}

\newcommand{\Ad}{\mathrm{Ad}}

\newcommand{\Aut}{\mathrm{Aut}}

\newcommand{\Trace}{\mathrm{Tr}}
\newcommand{\Haus}{\mathrm{H}}

\def\combplus[#1,#2,#3,#4]{\binom{#1\ {\scriptstyle #4} }{#2\ #3}}

\def\singlescalegenvert[#1,#2]{\hat{H}^{#2}_{#1}}
\def\multiscalegenvert[#1,#2]{H^{#2}_{#1}}

\def\nr[#1]{\tilde{N}[#1]} 
\def\inn[#1]{\mathring{N}[#1]}
\def\nrinn[#1]{\hat{N}_{#1}} 
\def\nrmod[#1,#2]{\tilde{N}_{#1}(#2)}
\def\nrinnmod[#1,#2]{\hat{N}_{#1}(#2)}

\def\ident[#1]{\underline{#1}}

\def\mylink#1#2{\mathrel{\vbox{\offinterlineskip\ialign{%
    \hfil##\hfil\cr
    $\scriptscriptstyle#1$\cr
    \noalign{\kern0.1ex}
    $#2$\cr
}}}}
\def\mysublink[#1]#2#3{\mathrel{\vbox{\offinterlineskip\ialign{%
    \hfil##\hfil\cr
    $\scriptscriptstyle#2$\cr
    \noalign{\kern0.1ex}
    $#3$\cr
    \noalign{\kern-0.2ex}
    \smash{\raisebox{-\height}{\hbox{$\scriptscriptstyle #1$}}}\cr
    \noalign{\kern0.2ex}
}}}}

\def\fon[#1]{\cC_{#1}}

\def\mincompproj[#1]{\mfp_{#1}}

\def\Proj_#1{\mathop{\mathrm{Proj}_{#1}}}

\def\negrenorm[#1]{\mfR_{#1}}
\def\topnegrenorm[#1]{\overline{\mfR}_{#1}}

\def\quotedge[#1]{E^{q}_{#1}}

\def\posrenorm[#1]{\mcC_{#1}}
\def\topposrenorm[#1]{\overline{\mcC_{#1}}}
\def\cutsmod[#1]{\mathbb{C}_{+,#1}}

\def\fullcutsmod[#1]{\cut_{#1}}

\colorlet{symbols}{black!50}
\colorlet{testcolor}{green!60!black}
\colorlet{darkblue}{blue!60!black}
\colorlet{darkgreen}{green!60!black}

\colorlet{redkernel}{red!80}

\def\symbol#1{\textcolor{symbols}{#1}}
\def\1{\mathbf{\symbol{1}}}

\usetikzlibrary{shapes.misc}
\usetikzlibrary{shapes.symbols}
\usetikzlibrary{shapes.geometric}
\usetikzlibrary{decorations}
\usetikzlibrary{decorations.markings}

\usetikzlibrary{calc}

\newcommand*{\mathcolor}{}
\def\mathcolor#1#{\mathcoloraux{#1}}
\newcommand*{\mathcoloraux}[3]{%
  \protect\leavevmode
  \begingroup
    \color#1{#2}#3%
  \endgroup
}

\makeatletter
\pgfdeclareshape{crosscircle}
{
  \inheritsavedanchors[from=circle] 
  \inheritanchorborder[from=circle]
  \inheritanchor[from=circle]{north}
  \inheritanchor[from=circle]{north west}
  \inheritanchor[from=circle]{north east}
  \inheritanchor[from=circle]{center}
  \inheritanchor[from=circle]{west}
  \inheritanchor[from=circle]{east}
  \inheritanchor[from=circle]{mid}
  \inheritanchor[from=circle]{mid west}
  \inheritanchor[from=circle]{mid east}
  \inheritanchor[from=circle]{base}
  \inheritanchor[from=circle]{base west}
  \inheritanchor[from=circle]{base east}
  \inheritanchor[from=circle]{south}
  \inheritanchor[from=circle]{south west}
  \inheritanchor[from=circle]{south east}
  \inheritbackgroundpath[from=circle]
  \foregroundpath{
    \centerpoint%
    \pgf@xc=\pgf@x%
    \pgf@yc=\pgf@y%
    \pgfutil@tempdima=\radius%
    \pgfmathsetlength{\pgf@xb}{\pgfkeysvalueof{/pgf/outer xsep}}%
    \pgfmathsetlength{\pgf@yb}{\pgfkeysvalueof{/pgf/outer ysep}}%
    \ifdim\pgf@xb<\pgf@yb%
      \advance\pgfutil@tempdima by-\pgf@yb%
    \else%
      \advance\pgfutil@tempdima by-\pgf@xb%
    \fi%
    \pgfpathmoveto{\pgfpointadd{\pgfqpoint{\pgf@xc}{\pgf@yc}}{\pgfqpoint{-0.707107\pgfutil@tempdima}{-0.707107\pgfutil@tempdima}}}
    \pgfpathlineto{\pgfpointadd{\pgfqpoint{\pgf@xc}{\pgf@yc}}{\pgfqpoint{0.707107\pgfutil@tempdima}{0.707107\pgfutil@tempdima}}}
    \pgfpathmoveto{\pgfpointadd{\pgfqpoint{\pgf@xc}{\pgf@yc}}{\pgfqpoint{-0.707107\pgfutil@tempdima}{0.707107\pgfutil@tempdima}}}
    \pgfpathlineto{\pgfpointadd{\pgfqpoint{\pgf@xc}{\pgf@yc}}{\pgfqpoint{0.707107\pgfutil@tempdima}{-0.707107\pgfutil@tempdima}}}
  }
}
\makeatother

\definecolor{connection}{rgb}{0.7,0.1,0.1}

\tikzset{
root/.style={circle,fill=black!50,inner sep=0pt, minimum size=3mm},
        dot/.style={circle,fill=black,inner sep=0pt, minimum size=1.2mm},
        sdot/.style={circle,fill=black,inner sep=0pt,minimum size=.5mm},
        dotred/.style={circle,fill=black!50,inner sep=0pt, minimum size=2mm},
        var/.style={circle,fill=black!10,draw=black,inner sep=0pt, minimum size=3mm},
        kernel/.style={semithick,shorten >=2pt,shorten <=2pt},
        kernel1/.style={thick},
        kernels/.style={snake=zigzag,shorten >=2pt,shorten <=2pt,segment amplitude=1pt,segment length=4pt,line before snake=2pt,line after snake=5pt,},
        rho/.style={densely dashed,semithick,shorten >=2pt,shorten <=2pt},
           testfcn/.style={dotted,semithick,shorten >=2pt,shorten <=2pt},
           tau/.style={circle,inner sep=1pt,draw=black,fill=white,text=black,thin},
        renorm/.style={shape=circle,fill=white,inner sep=1pt},
        labl/.style={shape=rectangle,fill=white,inner sep=1pt},
        xic/.style={very thin,circle,fill=symbols,draw=black,inner sep=0pt,minimum size=1.2mm},
        xi/.style={very thin,circle,fill=blue!10,draw=black,inner sep=0pt,minimum size=1.2mm},
        xix/.style={crosscircle,fill=blue!10,draw=black,inner sep=0pt,minimum size=1.2mm},
	xib/.style={very thin,circle,fill=blue!10,draw=black,inner sep=0pt,minimum size=1.6mm},
	xie/.style={very thin,circle,fill=green!50!black,draw=black,inner sep=0pt,minimum size=1.6mm},
	xid/.style={very thin,circle,fill=symbols,draw=black,inner sep=0pt,minimum size=1.6mm},
	xibx/.style={crosscircle,fill=blue!10,draw=black,inner sep=0pt,minimum size=1.6mm},
	kernels2/.style={very thick,draw=connection,segment length=12pt},
	not/.style={thin,circle,fill=symbols,draw=connection,fill=connection,inner sep=0pt,minimum size=0.5mm},
	>=stealth,
  }

\colorlet{darkblue}{blue!90!black}
\colorlet{darkred}{red!90!black}
\colorlet{darkgreen}{green!70!black}
\let\comm\comment

\def\${|\!|\!|}

\def\?{{\color{red}?}}

\def\slash{\kern0.18em/\penalty\exhyphenpenalty\kern0.18em}
\def\dash{\kern0.18em--\penalty\exhyphenpenalty\kern0.18em}

\newcommand{\floor}[1]{\lfloor #1 \rfloor}

\newtheorem{example}[lemma]{Example}

\let\basepoint\logof
\def\logof{\mathord{{\basepoint}}} 

\title{Yang--Mills measure on the two-dimensional torus as a random distribution}
\author{I.~Chevyrev}

\institute{University of Oxford\\ Email: chevyrev@maths.ox.ac.uk}

\date{}
\begin{document}
\maketitle
\begin{abstract}
We introduce a space of distributional $1$-forms $\Omega^1_\alpha$ on the torus $\T^2$ for which holonomies along axis paths are well-defined and induce H{\"o}lder continuous functions on line segments.
We show that there exists an $\Omega^1_\alpha$-valued random variable $A$ for which Wilson loop observables of axis paths coincide in law with the corresponding observables under the Yang--Mills measure in the sense of~\cite{Levy03}.
It holds furthermore that $\Omega^1_\alpha$ embeds into the H{\"o}lder--Besov space $\mathcal{C}^{\alpha-1}$ for all $\alpha\in(0,1)$, so that $A$ has the correct small scale regularity expected from perturbation theory.
Our method is based on a Landau-type gauge applied to lattice approximations.
\end{abstract}
\setcounter{tocdepth}{2}

\tableofcontents

\section{Introduction}

The main object of study in this paper is the Yang--Mills (YM) measure on the two-dimensional torus $\T^2$ given formally by
\begin{equ}\label{eq:YM}
\mrd \mu(A) = Z^{-1} e^{-S_\YM(A)} \mrd A\;.
\end{equ}
Here $\mrd A$ denotes a formal Lebesgue measure on the affine space $\mcA$ of connections on a principal $G$-bundle $P$ over $\T^2$, where $G$ is a compact, connected Lie group with Lie algebra $\mfg$.
For our purposes, we will always assume $P$ is trivial, so that after taking a global section, $\mcA$ can be identified with the space $\Omega^1(\T^2,\mfg)$ of $\mfg$-valued $1$-forms on $\T^2$.
The constant $Z$ is a normalisation which makes $\mu$ a probability measure, and the YM action $S_{\YM}(A)$ is defined by
\begin{equ}\label{eq:YM_Hamiltonian}
S_\YM(A) = \int_{\T^2} |F_A(x)|^2 \mrd x\;,
\end{equ}
where $F_A$ is the curvature two-form of $A$.

A number of authors with different techniques have investigated ways to give a rigorous meaning to~\eqref{eq:YM} (and its variants); a highly incomplete list is~\cite{BFS79, BS83, GKS89, Fine91, Sengupta97, Nguyen15}.
See also~\cite{Chatterjee18} for an extensive review on the literature associated with this problem.

One way to understand the measure is to study the distributions of certain gauge invariant observables.
A popular class of such observables are Wilson loops defined via holonomies, and a complete characterisation of these distributions can be found in~\cite{Levy03}, with related work going back to~\cite{Migdal75, DM79, Bralic80, Driver89, Witten91}.
We shall follow~\cite{Levy03,Levy10} and treat the YM measure as a stochastic process indexed by sufficiently regular loops in $\T^2$.

The purpose of this work is to realise the YM measure as a random distribution with the small scale regularity one expects from perturbation theory.
We show that a Landau-type gauge applied to lattice approximations allows one to construct a (non-unique) random variable taking values in a space of distributional $1$-forms for which a class of Wilson loops is canonically defined and has the same joint distributions as under the YM measure.

\subsection*{Outline of results}

The main result of this paper can be stated as follows (we explain the notation after the theorem statement).

\begin{theorem}\label{thm:main_thm}
Let $G$ be a compact, connected, simply connected Lie group with Lie algebra $\mfg$.
For all $\alpha \in (\frac12,1)$, there exists an $\Omega^1_\alpha(\T^2,\mfg)$-valued random variable $A$ such that for any $x\in\T^2$, finite collection of axis loops $\gamma_1,\ldots, \gamma_n$ based at $x$, and $\Ad$-invariant function $f:G^n\to\R$, it holds that $f(\hol(A,\gamma_1), \ldots,\hol(A,\gamma_n))$ is equal in law to $f$ applied to the corresponding holonomies under the YM measure.
\end{theorem}

An axis path is a piecewise smooth curve $\gamma : [0,1] \to \T^2$ formed by concatenating a finite number of paths of the form $t \mapsto x \pm t e_\mu$ for $x\in\T^2$ and $\mu\in\{1,2\}$.
The Banach space $(\Omega^1_\alpha(\T^2,\mfg), |\cdot|_\alpha)$ is defined in Section~\ref{subsec:additive_funcs_one_forms} and consists of distributional $\mfg$-valued $1$-forms.
The main feature of this space is that axis paths integrate along elements of $\Omega^1_\alpha$ to $\alpha$-H{\"o}lder paths in $\mfg$, which in turn can be developed into $G$ by Young integration.
For smooth $A$, $\hol(A,\gamma)$ is simply $y(1)$ where $y:[0,1]\to G$ solves the ODE
\begin{equ}
y'(t) = y(t)A(\gamma(t))[\gamma'(t)]\;,\quad y(0)=1_G\;,
\end{equ}
i.e., $y$ is the development into $G$ of the $\mfg$-valued path $\int_0^\cdot A(\gamma(t))[\gamma'(t)]\mrd t$.
The space $\Omega^1_\alpha$ further encodes regularity of $\hol(A,\gamma)$ as a function of $\gamma$.
For instance, denoting by $d$ the geodesic distance on $G$, if $\gamma,\bar\gamma$ parametrise parallel line segments at distance $\delta$, then $d(\hol(A,\gamma),\hol(A,\bar\gamma)) \leq C|A|_{\alpha} \delta^\kappa$ for some $C,\kappa>0$ depending only on $\alpha$.
Finally, a function $f:G^n\to\R$ is said to be $\Ad$-invariant if for all $h,g_1,\ldots, g_n\in G$
\begin{equ}
f(hg_1h^{-1},\ldots,hg_nh^{-1})=f(g_1,\ldots,g_n)\;.
\end{equ}
The class of all functions $A\mapsto f(\hol(A,\gamma_1), \ldots,\hol(A,\gamma_n))$, where $f$ is $\Ad$-invariant, is known to uniquely determine $A$ up to gauge equivalence (at least for smooth $A$), see~\cite[Prop.~2.1.2]{Sengupta92}.
This class includes the Wilson loop observables, i.e., functions which depend only on $\Trace[\phi \, \hol(A,\gamma_1)], \ldots, \Trace[\phi \, \hol(A,\gamma_n)]$ where $\phi$ is any finite-dimensional representation of $G$, but in general this class is strictly larger.

The article, as well as the proof of Theorem~\ref{thm:main_thm}, which is given at the end of Section~\ref{sec:prob_bounds}, is split into three parts.
The first part, given in Section~\ref{sec:hol_dist}, constructs the space $\Omega^1_\alpha$ and derives its basic properties.
In this part we work in arbitrary dimension $d \geq 1$.
The second part, which can be seen as the main contribution of this paper, is given in Section~\ref{sec:det_bounds} and defines a gauge on lattice approximations through iterations of the Landau gauge $\sum_{\mu=1}^d \partial_\mu A_\mu = 0$ (also called the Coulomb gauge in differential geometry).
We furthermore apply an axial gauge in order to reach a small $1$-form on some medium scale, after which the preceding gauge can be applied.
The third part, given in Section~\ref{sec:prob_bounds}, again uses an axial-type gauge together with a random walk argument to obtain probabilistic bounds necessary to apply the results from Section~\ref{sec:det_bounds}.
We work with quite general discrete approximations as in~\cite[Sec.~7]{Driver89} which cover the Villain (heat kernel) and Wilson actions.

\begin{remark}
The assumption that $G$ is simply connected appears for topological reasons when applying the axial gauge in Section~\ref{subsec:axial_gauge} (and would not be necessary if we worked on the square $[0,1]^2$ instead of $\T^2$).
In fact, one does not expect to be able to represent a realisation of the YM holonomies as a global $1$-form unless the realisation is associated to a trivial principal bundle.
How to construct the YM measure associated to a specific principal bundle was understood in~\cite{Levy06}, and it would be of interest to extend our results to this general case.
\end{remark}

\begin{remark}\label{rem:more_paths}
The restriction to axis paths appears superficial, and is certainly an artefact of our proof.
The construction in~\cite{Levy03} makes sense of the corresponding random variables for any piecewise smooth embeddings $\gamma_i$, and this was later extended to all bounded variation paths in~\cite{Levy10}.
It would be of interest to determine a more canonical space of ``test'' paths in our context for which $\hol(A,\gamma)$ is well-defined together with regularity estimates.
The construction in Section~\ref{sec:hol_dist} could be adapted to different classes of paths, however it is unclear how to adapt the results of Section~\ref{sec:det_bounds} and~\ref{sec:prob_bounds} to yield a satisfactory conclusion.
See also Remark~\ref{rem:BV_paths}.
\end{remark}

The Landau-type gauge defined in Section~\ref{subsec:Landau_gauge} can be loosely explained as follows: we first apply the classical Landau gauge on low dimensional subspaces, working up to the full dimension (for $d=2$ this involves just two steps), and then propagate the procedure from large to small scales.
The advantage of this gauge is that it is relatively simple to analyse and retains the small scale regularity expected from perturbation theory (which is not true, e.g., for the axial gauge).
The exact form of this gauge appears new (although it is closely related to the classical Landau gauge, which is of course well-known) and its regularity analysis can be seen as the main technical contribution of this paper.
We choose to study this gauge only in dimension $d=2$ since this simplifies many arguments, and since this restriction is crucial for our probabilistic estimates, however we emphasise that an analogous construction works in arbitrary dimension.
See Remarks~\ref{rem:poisson} and~\ref{rem:rand_walk} for the intuition behind this gauge coming from elliptic PDEs. 

While we work with approximations of the YM measure taken from~\cite{Levy03,Driver89}, we note that our analysis is closer in spirit to that of~\cite{Balaban85I,Balaban85II,Balaban85III} (which was subsequently used to prove ultraviolet stability of three- and four-dimensional lattice approximations of the pure YM field theory under the action of a renormalisation group).

\subsection*{Motivation and further directions}

It would be of interest to extend our work to higher dimensions to yield small scale regularity of lattice approximations to the YM measure in $d=3$.
See~\cite{Chatterjee16} for recent work on the YM measure in three and four dimensions.
The difficulty here is of course that the measure becomes much more singular and requires non-trivial renormalisation.
Furthermore, one does not necessarily expect from perturbation theory that Wilson loop observables would be well-defined even for $d=3$ (see Remark~\ref{rem:dim_three} and~\cite[Sec.~3.1]{CG15},~\cite[Sec.~3]{Frohlich80}).
In this case one may need to regularise the connection as propsed in~\cite{CG13,CG15} or consider smooth averages of Wilson loops, see e.g.~\cite[p.~819]{Singer81}.
Another direction would be to work with so-called lasso variables~\cite{Gross85,Driver89} which could prove more regular in higher dimensions than Wilson loops.

We end the introduction with a discussion on one of the motivations behind this paper.
An important feature of the space $\Omega^1_\alpha$ is its embedding into $\Omega^1_{\mcC^{\alpha-1}}$, the space of H{\"o}lder--Besov distributions commonly used in analysis of stochastic PDEs~\cite{Hairer14,GIP15}, see Corollary~\ref{cor:cont_embed}.
The main result of this paper can thus be seen as a construction of a candidate invariant measure (up to suitable gauge transforms) for the connection-valued stochastic YM heat flow
\begin{equs}
\partial_t A &= -\mrd_A^* F_A -  \mrd_A\mrd_A^* A + \xi\label{eq:SYM}
\\
&= \sum_{\mu=1}^d \mrd x_\mu \Big(\Delta A_\mu +\xi_\mu + \sum_{\nu=1}^d [A_\nu,2\partial_\nu A_\mu - \partial_\mu A_\nu] + [A_\nu,[A_\nu,A_\mu]] \Big)\;,
\end{equs}
where $\mrd_A$ is the covariant derivative, $F_A$ is the curvature two-form of $A$, and $\xi$ is a space-time white noise built over the Hilbert space $\Omega^1(\T^2,\mfg)$, i.e., $(\xi_\mu)_{\mu=1}^d$ are iid $\mfg$-valued space-time white noises.
The term $\mrd_A\mrd_A^* A$, known as the DeTurck~\cite{deturck83} or Zwanziger~\cite{zwanziger81} term, is a gauge breaking term which renders the equation parabolic (and the solution gauge equivalent to the solution without this term).

The YM heat flow without noise is a classical tool in geometry~\cite{DK90}; for a recent application, see~\cite{Oh14,Oh15} where the deterministic YM heat flow was applied to establish well-posedness of the YM equation in Minkowski space.
It was also proposed in~\cite{CG13} as a gauge invariant continuum regularisation of rough connections; one of the motivations therein was to set up a framework in which one could define a non-linear distributional (negative index Sobolev) space which could support the YM measure for non-Abelian gauge groups (a goal which parallels the one of this article).

The motivation to study the stochastic dynamics arises from stochastic quantization~\cite{DH87, BHST87I}.
The principle idea is to view~\eqref{eq:SYM} as the Langevin dynamics for the Hamiltonian~\eqref{eq:YM_Hamiltonian} of the YM model.
This quantization procedure largely avoids gauge fixing, the appearance of Faddeev--Popov ghosts, and the Gribov ambiguity, which was one of the motivations for its introduction by Parisi--Wu~\cite{ParisiWu}.
It was furthermore recently used to rigorously construct the scalar $\Phi^4_3$ measure on the torus~\cite{MW17Phi43}.

Due to the roughness of the noise $\xi$ and the non-linearity of the term $\mrd_A^* F_A$ in the non-Abelian case, equation~\eqref{eq:SYM} is classically ill-posed.
The framework of regularity structues~\cite{Hairer14, CH16, BHZ16, BCCH17} however provides an automated local solution theory for this equation in dimension $d<4$ (at least via smooth mollifier approximations).
Shen~\cite{Shen18} recently studied lattice approximations of the Abelian version of this equation coupled with a Higgs field using discretizations of regularity structures~\cite{EH17, HM18, CM18}.
One also expects the equation to be amenable to paracontrolled analysis and its discretizations~\cite{GIP15, GP17, MP17, ZZ18}.

\begin{remark}
Another way to construct the YM measure as a random distribution is through the axial gauge~\cite{Driver89}.
One can verify however that this construction yields a random distribution of regularity $\mcC^{\eta}$ for $\eta<-\frac{1}2$ and that the procedure in~\cite{Hairer14, BCCH17} yields a solution theory for~\eqref{eq:SYM} only for initial conditions in $\mcC^\eta$ for $\eta > -\frac{1}{2}$.
\end{remark}

In a similar way to~\cite{HM18}, one could expect that~\eqref{eq:SYM} admits global in time solutions for a.e. starting point from an invariant measure.
In addition to~\cite{LevyNorris06}, where a large deviations principle is shown, such a result would provide a further rigorous link between the YM measure and the YM energy functional.

\begin{remark}
Note that the term $\mrd_A\mrd_A^*A$ acts as a globally restoring force, and the (formal) stable fixed points of $\partial_t A = -\mrd_A\mrd_A^*A$ are the connections satisfying the Landau gauge $\sum_{\mu=1}^d \partial_\mu A_\mu = 0$ within the so-called Gribov region (which is well-known in the physics literature, see~\cite[Sec.~4.2.2]{DH87} or~\cite[Sec.~4.5]{VZ12}).
It is therefore possible that global in time solutions could exist a.s. for arbitrary initial conditions, but it is unclear if this should be expected.
This is true for the $\Phi^4$ models~\cite{MW17Phi42,MW17Phi43}, though through a rather different mechanism.
Global in time stability of the YM heat flow without noise is already somewhat non-trivial, even in $d=2,3$~\cite{Rade92}, and typically uses Uhlenbeck compactness~\cite{Uhlenbeck82, Wehrheim04}.
\end{remark}

\subsection*{Acknowledgements}
The author is grateful to Thierry L{\'e}vy for several clarifying remarks which helped identify an error in a previous version of the paper.
The author also thanks the anonymous referee for helpful comments.
The author is funded by a Junior Research Fellowship of St John's College, Oxford.

\section{Notation and conventions}\label{subsec:notation}

\subsection{Paths} 

For a set $E$ and a function $\gamma:[0,1]\to E$, we denote by $\gamma_{[0,1]} \subset E$ the image of $\gamma$.
For a metric space $(E,d)$, $q \geq 1$, and a path $\gamma : [s,t] \to E$, we define the $q$-variation of $\gamma$ by
\[
|\gamma|_{\var q} \eqdef \sup_{D\subset [s,t]}\Big(\sum_{t_i\in D} d(\gamma(t_i),\gamma(t_{i+1}))^q \Big)^{1/q}\;,
\]
where the supremum is taken over all finite partitions $D = (s \leq t_0 < t_1 < \ldots < t_{n} \leq t)$ (with $t_{n+1} \eqdef t$ for the case $t_i=t_n$ in the sum above).
For a sequence $(\gamma(i))_{i=1}^k$ with $\gamma(i) \in E$, we denote by $|\gamma|_{\var q}$ the same quantity with the supremum taken over all subsequences $D= (t_0 < \ldots < t_n )$ of $\{1,\ldots, k\}$ (this time with $t_{n+1}\eqdef k$ in the sum above).
We denote by $\mcC^{\var q}([s,t],E)$ the set of continuous paths $\gamma : [s,t] \to E$ for which $|\gamma|_{\var q} < \infty$.
Similarly, for $\alpha \in [0,1]$, we let $\mcC^{\Hol\alpha}([s,t],E)$ denote the set of paths $\gamma : [s,t] \to E$ for which
\[
|\gamma|_{\Hol \alpha} \eqdef \sup_{s\leq u < v \leq t}\frac{d(\gamma(u),\gamma(v))}{|v-u|^\alpha} < \infty\;.
\]

\subsection{Lattices}

For an integer $d \geq 1$, we set $[d]\eqdef \{1,\ldots, d\}$.
Let $(e_{\mu})_{\mu=1}^d$ be an orthonormal basis of $\R^d$ and let $\Z^d$ denote the lattice generated by $(e_{\mu})_{\mu=1}^d$.
We will work primarily on the torus $\T^d \eqdef \R^d/\Z^d$ equipped with its usual (geodesic) metric which, by an abuse of notation, we denote by $|x-y|$.
As a set, we will identify $\T^d$ with $[0,1)^d$ in the usual way
and write $x=(x_1,\ldots, x_d)$ for $x\in\T^d$.

Let $\pi_{\T^d} : \R^d \to \T^d$ denote the canonical projection.
For $N \geq 0$, we define the lattice $\Lambda_N \eqdef \pi_{\T^d}2^{-N}\Z^d$,\label{page ref Lambda} which we identify with $\{0,2^{-N},\ldots, (2^{N}-1)2^{-N}\}^d$ as a set.
We say that $x,y\in\Lambda_N$ are adjacent if $|x-y| = 2^{-N}$.
An \emph{oriented bond}, or simply \emph{bond}, of $\Lambda_N$ is an ordered pair of adjacent points $\alpha = (x,x\pm 2^{-N}e_{\mu})\in\Lambda_N^2$ where $\mu \in [d]$.
We call $\overleftarrow{\alpha} = (x\pm 2^{-N}e_\mu,x)$ the \emph{reversal} of $\alpha$.
We denote by $\Bonds_N$\label{page ref Bonds} the set of bonds of $\Lambda_N$.
We further denote by $\overline{\Bonds}_N$\label{page ref overline Bonds} the subset of bonds $(x,x+2^{-N}e_\mu)\in\Bonds_N$.
Note that every $\alpha \in \Bonds_N$ canonically defines a subset of $\T^d$ with one-dimensional Lebesgue measure $|\alpha| \eqdef 2^{-N}$, and that $\alpha,\bar\alpha\in\Bonds_N$ define the same subset of $\T^d$ if and only if $\bar\alpha=\alpha$ or $\bar\alpha = \overleftarrow{\alpha}$.
In the same way, we can canonically identify every $\alpha \in\overline\Bonds_N$ with a subset of $\T^d$.

A \emph{rectangle} of $\Lambda_N$ is a triplet $r=(x,m2^{-N}e_\mu, n2^{-N}e_\nu)$\label{page ref rectangle} where $x \in \Lambda_N$, $1 \leq \mu < \nu \leq d$, and $1 \leq m,n < 2^N$ with either $m=1$ or $n=1$.
Observe that $r$ can be canonically identified with a subset of $\Lambda_N$ consisting of $(m+1)(n+1)$ points, as well as a (closed) subset of $\T^d$ with two-dimensional Lebesgue measure $|r| = mn 2^{-2N}$.
We will freely interchange between these interpretations.
If $m=n=1$, we call $r$ a plaquette.\label{page ref plaquette}

We let $\Grid_N \subset \T^d$\label{page ref Grid} denote the grid induced by $\Lambda_N$, that is,
\[
\Grid_N \eqdef \{x+ce_\mu \ssep x \in \Lambda_N, c \in [0,1], \mu \in [d] \}\;.
\]

\subsection{\texorpdfstring{$1$}{1}-forms and gauge fields}
\label{subsec:one-forms}

For a vector space $E$, we let $\Omega^{1,(N)} = \Omega^{1,(N)}(\T^d,E)$\label{page ref Omega1N} denote the space of functions $A : \Bonds_N \to E$ such that $A(\alpha) = -A(\overleftarrow{\alpha})$.
We call elements of $\Omega^{1,(N)}$ discrete $E$-valued $1$-forms on $\Lambda_N$.
Note that for $\bar N \leq N$, every $A\in\Omega^{1,(N)}$ canonically defines a function $A \in \Omega^{1,(\bar N)}$ (which we denote by the same letter) via
\begin{equ}\label{eq:restrict_A}
A(x,x+2^{-\bar N}e_\mu) \eqdef \sum_{k=0}^{2^{N-\bar N}-1} A(x+k2^{-N}e_\mu, x+(k+1)2^{-N}e_\mu)\;.
\end{equ}
We will often use the shorthand $A^{\bar N}_\mu(x) \eqdef A(x,x+2^{-\bar N}e_\mu)$.

Throughout the paper we let $G$ \label{page ref G} be a compact, connected Lie group (not necessarily simply connected)
with Lie algebra $\mfg$.\label{page ref mfg}
We let $1_G$ denote the identity element of $G$.
We equip $G$ with the normalised Haar measure denoted in integrals by $\mrd x$.
We equip $\mfg$ with an $\Ad(G)$ invariant inner product $\scal{\cdot,\cdot}$ and equip $G$ with the corresponding Riemannian metric and geodesic distance.
We fix a measurable map $\log : G \to \mfg$ \label{page ref log} with bounded image such that $\exp(\log x) = x$ for all $x \in G$ and such that $\log$ is a diffeomorphism between a neighbourhood of $1_G$ and a neighbourhood of $0 \in \mfg$.
We further choose $\log$ so that $\log(yxy^{-1})= \Ad_y \log x$ for all $x,y\in G$ and $\log(x) = -\log(x^{-1})$ for all $x \in G$ outside a null-set (this is always possible by considering a faithful finite-dimensional representation of $G$ and the principal logarithm, cf.~\cite[Sec.~A]{Balaban85I}; the last point follows from the fact if $G$ is a compact, connected matrix group, then $\{x\in G \ssep -1 \in \sigma(x)\}$ has Haar measure zero --- this is obvious if $G$ is Abelian, and the general case follows e.g. from the Weyl integral formula~\cite[Thm.~11.30]{Hall15}).

\begin{remark}
In the sequel, when we say that a quantity depends on $G$, we implicitly mean it depends also on the choice of $\log$ and inner product on $\mfg$.
\end{remark}

We denote by $\mfA^{(N)}$\label{page ref mfA} the set of functions $U:\Bonds_N \to G$ such that $U(\alpha) = U(\overleftarrow{\alpha})^{-1}$.
Observe that every $A \in \Omega^{1,(N)}(\T^d,\mfg)$ defines an element of $\mfA^{(N)}$ via $U = \exp A$.
Note further that every $U \in \mfA^{(N)}$ canonically defines an element in $\mfA^{(\bar N)}$ for all $\bar N \leq N$ exactly as in~\eqref{eq:restrict_A} with the sum replaced by an ordered product.
We will again often use the shorthand $U^{\bar N}_\mu(x) \eqdef U(x,x+2^{-\bar N}e_\mu)$.

We let $\mfG^{(N)}$\label{page ref mfG} denote the set of functions $g : \Lambda_N \to G$.
We call elements of $\mfG^{(N)}$ discrete gauge transforms.
For $U \in \mfA^{(N)}$ and $g \in \mfG^{(N)}$, we define $U^g \in \mfA^{(N)}$\label{page ref gauge transform} by
\[
U^g(x,y) \eqdef g(x)U(x,y)g(y)^{-1}\;.
\]

We define the binary power of a number $q \in [0,1)$ as the smallest $k \geq 0$ such that $q = \sum_{i=0}^k \lambda_i 2^{-i}$ with $\lambda_i \in \{0,1\}$
(if no such $k$ exists, then the binary power of $q$ is $\infty$).
For a plaquette $p = (x,2^{-N}e_\mu,2^{-N}e_\nu)$, note that there is a unique $z = (z_1,\ldots, z_d) \in p \cap \Lambda_N$ such that $z_\mu$ and $z_\nu$ have binary power at most $N-1$, and for the other three points $y \in p \cap\Lambda_N$, at least one of $y_\mu,y_\nu$ has binary power $N$.
We call $z$ the origin of $p$.
For $U \in \mfA^{(N)}$, we define $U(\partial p) \eqdef U(\alpha_1)\ldots U(\alpha_4)$ where $\alpha_1,\ldots, \alpha_4$ are the four bonds oriented to traverse the boundary of $p$ anti-clockwise starting at $z$ when viewed from the $(\mu,\nu)$ plane.

\begin{example}
Consider $d=3$, $N=2$, $\mu=1,\nu=3$, and $x = (\frac14,\frac14,\frac12)$.
Then the origin of $p \eqdef (x,2^{-2}e_1,2^{-2}e_3)$ is $z = x+2^{-2}e_1 = (\frac12,\frac14,\frac12)$ and
\begin{equs}
\alpha_1 &= (z,z_1) \eqdef (z,z+2^{-2}e_3)\;,
\\
\alpha_2 &= (z_1,z_2) \eqdef (z_1, z_1-2^{-2}e_1)\;,
\\
\alpha_3 &= (z_2,z_3) \eqdef (z_2,x) = (z_2, z_2-2^{-2}e_3)\;,
\\
\alpha_4 &= (z_3,z) =(x,z) = (z_3,z_3+2^{-2}e_1)\;.
\end{equs}
\end{example}
In general, for a rectangle $r = (x,m2^{-N}e_\mu,n2^{-N}e_\nu)$, there is a unique plaquette $p \subset r$ such that neither $p-2^{-N}e_\mu$ nor $p-2^{-N}e_\nu$ are contained in $r$.
We define the origin $z$ of $r$ as the origin of $p$, and define $U(\partial r) \eqdef U(\alpha_1)\ldots U(\alpha_k)$\label{page ref U partial r} where $\alpha_1,\ldots, \alpha_k$ are the bonds in $\Bonds_N$ which traverse the boundary of $r$ anti-clockwise starting from $z$ when viewed from the $(\mu,\nu)$ plane.

\begin{remark}
The exact order of the bonds $\alpha_i$ may seem arbitrary at this point (one usually simply starts at the south-west corner of $r$), but this choice will be convenient in Section~\ref{subsec:Landau_gauge}.
\end{remark}

\section{Holonomy on distributions}
\label{sec:hol_dist}

In this section we introduce spaces of distributional $1$-forms on $\T^d$ for which integration along axis paths is canonically defined.
We will later show that the YM measure can be appropriately gauged fixed to have support on these spaces.

\subsection{Motivation: the Gaussian free field}

From perturbation theory, we expect that in two and three dimensions the YM measure can be realised as a random distribution with the same regularity as the Gaussian free field (GFF) $\Psi$.
In this subsection, we present an informal discussion about what precisely we mean by ``regularity''.

Working on $\T^2$, it is well-known that $\Psi$ is not a function (though it is almost a function since it belongs to every H{\"o}lder--Besov space $\mcC^{-\kappa}$, $\kappa>0$).
Pointwise evaluation $\Psi(x) = \scal{\Psi,\delta_x}$ is therefore ill-defined.
We claim however, that for certain regular curves $\gamma:[0,1] \to \T^2$, the integral $\scal{\Psi,\gamma} = \int_0^1 \Psi(\gamma(t))\gamma'(t) \mrd t$ is canonically defined.

Consider a straight line segment $\ell = \{x+ty \ssep t\in[0,1]\}$ where $x\in\T^2$ and $y\in\R^2$
with length $|\ell| \eqdef |y| < 1$.
For $\psi\in\mcC(\T^2)$, we define the Dirac delta $\delta_\ell$ by $\scal{\psi,\delta_\ell} \eqdef \int_0^1 \psi(x+ty) \mrd t$.
Denoting $\gamma(t) = x+ty$, observe that if $A_1,A_2 : \T^2 \to \R$ are bounded, measurable functions, so that $A = \sum_{\mu=1}^2 A_\mu \mrd x_\mu \in \Omega^1(\T^2, \R)$ is a bounded, measurable $\R$-valued $1$-form, then the integral
\begin{equ}\label{eq:A_int}
A(\gamma) \eqdef \int_0^1 \sum_{\mu=1}^2 A_\mu(\gamma(t))\gamma_\mu'(t) \mrd t
\end{equ}
is given by $\scal{\psi,\delta_\ell}$ for $\psi$ a suitable linear combination of $A_1,A_2$.

The point here is that $\scal{\psi,\delta_\ell}$ can make sense for sufficiently regular distributions $\psi$.
Specifically, writing $K$ for the convolution kernel of $\Delta^{-1/2}$, we have $|K(x)| \sim |x|^{-1}$, and thus $|K*\delta_\ell(x)| \sim |\log d(x,\ell)|$, where $d(x,\ell) = \inf\{|x-z| \ssep z \in \ell\}$.
Hence $\Delta^{-1/2}\delta_\ell$ is a function in $L^2$ (with plenty of room to spare) and the evaluation $\scal{\Psi,\delta_\ell}$ makes sense (as a random variable) where $\Psi=\Delta^{-1/2}\xi$ is a GFF and $\xi$ is an $\R$-valued white noise on $\T^2$. 

\begin{remark}\label{rem:dim_three}
Note that the same is not true in three dimensions.
In this case $K(x) \sim |x|^{-2}$ so that $K*\delta_\ell(x) \sim |d(x,\ell)|^{-1}$, rendering the integral $\int |K*\delta_\ell(x)|^2 \mrd x$ infinite (but only just).
This suggests that, even in the smoothest gauge, Wilson loops would a.s. not be defined for the YM measure in dimension three, cf.~\cite[p.~160]{BFS80}.
We note however, that replacing $\ell$ by a suitable surface $L$ again renders $K*\delta_L(x) \sim |\log d(x,L)|$ so that $\Delta^{-1/2}\delta_L$ is in $L^2$ (with plenty of room to spare).
\end{remark}

Furthermore, one can derive growth bounds and H{\"o}lder continuity with respect to $\ell$.
To see this, note that $|(K*\delta_\ell)(x)| \lesssim \log(d(x,\ell)+|\ell|)-\log d(x,\ell)$, from which it follows that $|K*\delta_\ell|^2_{L^2} \lesssim |\ell|^{2\alpha}$ for any $\alpha<1$ (e.g. by splitting the domain of integration into annuli around $\ell$ with radii $|\ell|2^N$).
Hence $\scal{\Psi,\delta_\ell}$ is a Gaussian random variable with variance $\lesssim |\ell|^{2\alpha}$.
Similarly, if $\ell,\bar\ell$ are parallel line segments at distance $d(\ell,\bar\ell)$, then, using $|\nabla K(x)| \sim |x|^{-2}$, one can show that $|K*(\delta_\ell-\delta_{\bar\ell})|_{L^2}^2 \lesssim |\ell|^{\alpha}d(\ell,\bar\ell)^{\alpha}$.
Hence $\scal{\Psi,\delta_\ell-\delta_{\bar\ell}}$ is a Gaussian random variable with variance $\lesssim |\ell|^{\alpha}d(\ell,\bar\ell)^{\alpha}$.
One can combine these two estimates in a Kolmogorov-type argument (at least for axis line segments) to show that, for any $\alpha < 1$,
\begin{equ}
|\scal{\Psi,\delta_\ell}| \lesssim |\ell|^{\alpha} \quad \textnormal{ and } \quad |\scal{\Psi,\delta_\ell -\delta_{\bar\ell}}| \lesssim |\ell|^{\alpha/2}d(\ell,\bar\ell)^{\alpha/2} \quad \textnormal{ a.s.}\;.
\end{equ}
(A more precise formulation would be that $\Psi$ admits a modification for which these bounds holds.)

Sections~\ref{sec:det_bounds} and~\ref{sec:prob_bounds} of this paper can be seen as deriving these estimates and Kolmogorov argument when $\Psi$ is replaced by discrete approximations of the YM measure (albeit with rather different methods).
The remainder of this section sets up the space in which we will obtain weak limit points of these approximations. 

\begin{remark}\label{rem:YM_hol}
The analogue for the YM measure $U$ (as a random holonomy) of the estimate $|\scal{\Psi,\delta_\ell-\delta_{\bar\ell}}| \lesssim |\ell|^{\alpha/2}d(\ell,\bar\ell)^{\alpha/2}$ takes the form $|\log U(\partial r)| \lesssim |r|^{\alpha/2}$ where $r$ is the rectangle with $\ell,\bar\ell$ as two of its sides.
This is certainly expected since the law of $U(\partial r)$ is close to that of $B_{|r|}$, where $B$ is a $G$-valued Brownian motion.
\end{remark}

\begin{remark}\label{rem:BV_paths}
We restrict attention in this article to axis line segments (and thus finite concatenations thereof).
It would be desirable to work with a more natural class of paths along which holonomies could be defined together with similar estimates, but it is not entirely clear what the correct ``test-space'' should be.
For example, if $A$ was a random $\mfg$-valued $1$-form which induced the YM holonomies, one would expect that for a.e. realisation there should exist a bounded variation path $\gamma$ for which $A(\gamma)$ defined by~\eqref{eq:A_int} does not exist (e.g., concatenations of small square loops rapidly decreasing in size but with an increasing number of turns around each one).
Thus it seems necessary to impose some control on the derivative of $\gamma$ for $A(\gamma)$ and $\hol(A,\gamma)$ to be well-defined \emph{pathwise} (cf. Remark~\ref{rem:more_paths}).
\end{remark}

\subsection{Functions on line segments}

We formalise the above discussion by introducing a suitable space of distributions.

\begin{definition}
We call a subset $\ell \subset \T^d$ an axis line segment if $\ell = \{x+c e_\mu \ssep c \in [0,\lambda]\}$ for some $x \in \T^d$, $\mu \in [d]$, and $\lambda \in [0,1]$.
In this case we define $|\ell| \eqdef\lambda$ and, if $|\ell| > 0$, we say that the direction of $\ell$ is $\mu$.
We let $\mcX$ \label{page ref mcX} denote the set of all axis line segments
equipped with the Hausdorff metric $d_\Haus$.\label{page ref dH}
\end{definition}

Note that $\mcX$ is a compact metric space.
We introduce another distance on $\mcX$.

\begin{definition}
For $\mu \in [d]$ let $\pi_\mu : \T^d \to \T$ denote the projection onto the $\mu$-th axis.
We say that $\ell,\bar\ell \in\mcX$ are parallel if they have the same direction $\mu \in [d]$ and $\pi_\mu \ell = \pi_\mu \bar \ell$.
For parallel $\ell,\bar\ell\in\mcX$ we define\label{rho page ref}
\[
\rho(\ell,\bar\ell) \eqdef |\ell|^{1/2} d(\ell,\bar\ell)^{1/2}
\]
where $d(\ell,\bar\ell) \eqdef \inf\{ |x-y| \ssep x\in \ell, y \in \bar\ell\}$.
\end{definition}

Note that $\rho(\ell,\bar\ell)^2$ is the area of the smallest rectangle with two of its sides as $\ell$ and $\bar\ell$.

For the rest of the section, let $E$ be a fixed finite-dimensional normed space.
\begin{definition}
We say that $\ell,\bar\ell \in \mcX$ are joinable if $\ell\cup\bar\ell \in \mcX$ and $|\ell\cup\bar\ell| = |\ell|+|\bar\ell|$.
We say that a function $A : \mcX \to E$ is additive if for all joinable $\ell,\bar\ell \in \mcX$, we have $A(\ell\cup\bar\ell) = A(\ell)+A(\bar\ell)$.
Let $\Omega$\label{Omega page ref} denote the space of all additive functions $A:\mcX\to E$.
\end{definition}

\begin{definition}\label{def:rhoHol}
For $A \in \Omega$ and $\alpha \in [0,1]$ we define
\begin{equ}
|A|_{\alpha;\rho} \eqdef \sup_{\ell \neq \bar\ell}\frac{|A(\ell)-A(\bar\ell)|}{\rho(\ell,\bar\ell)^{\alpha}}\;,
\end{equ}
where the supremum is taken over all distinct parallel $\ell,\bar\ell \in \mcX$.
We also define the $\alpha$-growth norm
\begin{equ}
|A|_{\gr\alpha} \eqdef \sup_{|\ell|>0}\frac{|A(\ell)|}{|\ell|^{\alpha}}\;.
\end{equ}
where the supremum is taken over all $\ell \in \mcX$ with $|\ell| > 0$.

Define $|\cdot|_\alpha \eqdef |\cdot|_{\gr\alpha} + |\cdot|_{\alpha;\rho}$ and let $\Omega_{\alpha}$\label{page ref Omega alpha} denote the Banach space $\{A \in\Omega \ssep |A|_{\alpha} < \infty\}$ equipped with the norm $|\cdot|_\alpha$.
\end{definition}

For $\ell \in \mcX$, we call a parametrisation of $\ell$ a path $\gamma : [0,1] \to \T^d$ with constant derivative $\gamma' \equiv |\ell|e_\mu$ such that $\gamma_{[0,1]} = \ell$.
Note that if $|\ell| < 1$, there is exactly one parametrisation of $\ell$.
For every $A\in\Omega$ and $\ell \in \mcX$ with $|\ell|<1$, one can canonically construct a path $\ell_A : [0,1] \to E$\label{page ref ellA} by
\begin{equ}
\ell_A(t) \eqdef A(\gamma_{[0,t]})\;,
\end{equ}
where $\gamma$ is the unique parametrisation of $\ell$.
We have the following basic result, the proof of which is obvious.

\begin{lemma}\label{lem:ellA_Hol_bound}
Let $\alpha \in [0,1]$, $\ell \in \mcX$ with $|\ell| < 1$, and $A \in \Omega$.
Then $|\ell_A|_{\Hol{\alpha}} \leq |\ell|^\alpha|A|_{\gr\alpha}$.
\end{lemma}


We show next that $|\cdot|_{\gr\alpha}$ and $|\cdot|_{\alpha;\rho}$ bound the $\frac\alpha2$-H{\"o}lder norm of $A$ with respect to $d_{\Haus}$.
\begin{proposition}\label{prop:d_Haus_cont}
Let $\alpha \in [0,1]$, $A \in \Omega$, and $\ell,\bar\ell\in\mcX$. Then
\[
|A(\ell)-A(\bar\ell)| \leq 2^{1-\alpha/2}|A|_{\alpha;\rho} (|\ell|\wedge|\bar\ell|)^{\alpha/2}d_\Haus(\ell,\bar\ell)^{\alpha/2} + 2^{1+\alpha}|A|_{\gr\alpha} d_{\Haus}(\ell,\bar\ell)^\alpha\;.
\]
\end{proposition}

We break the proof up into several elementary lemmas.

\begin{lemma}\label{lem:length_Haus_bound}
Suppose $\ell,\bar\ell\in\mcX$ do not have the same direction.
Then $|\ell| \leq 2d_\Haus(\ell,\bar\ell)$.
\end{lemma}

\begin{proof}
Let $\mu$ be the direction of $\ell$.
Then
\[
|\ell| = |\pi_\mu(\ell)| \leq 2d_\Haus(\pi_\mu\ell,\pi_\mu\bar\ell) \leq 2d_\Haus(\ell,\bar\ell)
\]
where in the first inequality we used that $\pi_\mu\bar\ell$ is a single point, and in the second inequality we used that $\pi_\mu : \T^d\to\T$ does not increase distance.
\end{proof}

Let $|X|$ denote the Lebesgue measure of a (measurable) subset $X\subset \T$, and let $X\triangle Y$ denote the symmetric difference of $X,Y\subset \T$.

\begin{lemma}\label{lem:symm_diff}
Let $X,Y$ be subsets of $\T$ each with a single connected component.
Then $|X\triangle Y| \leq 4d_\Haus(X,Y)$.
\end{lemma}

\begin{proof}
Clearly $X\triangle Y$ has at most two connected components and every connected component has Lebesgue measure at most $2d_\Haus(X,Y)$.
\end{proof}

Consider a pair $\ell,\bar\ell \in \mcX$ with the same direction $\mu \in [d]$.
It holds that $\pi_\mu\ell\cap\pi_\mu\bar\ell$ has at most two connected components which we call $X,Y$ (one or both possibly empty).
Likewise, $\pi_\mu\ell\triangle\pi_\mu\bar\ell$ has at most two connected components, which we call $U,V$ (one or both possibly empty).

\begin{lemma}\label{lem:overlap_Haus_bound}
Let notation be as in the preceding paragraph.
Then
\begin{enumerate}[label=(\alph*)]
\item\label{point:XY} $(|X|+|Y|)d(\ell,\bar\ell) \leq (|\ell|\wedge |\bar\ell|) d_\Haus(\ell,\bar\ell)$,
\item\label{point:UV} $|U|+|V| \leq 4 d_\Haus(\ell,\bar\ell)$.
\end{enumerate}
\end{lemma}

\begin{proof}
For~\ref{point:XY}, we have $|X|+|Y| \leq |\ell|\wedge |\bar\ell|$ and $d(\ell,\bar\ell) \leq d_\Haus(\ell,\bar\ell)$.
For~\ref{point:UV}, since $\pi_\mu:\T^d \to \T$ does not increase distance, we have $d_\Haus(\pi_\mu\ell,\pi_\mu\bar\ell) \leq d_\Haus(\ell,\bar\ell)$.
Hence, by Lemma~\ref{lem:symm_diff}, $|U|+|V| = |\pi_\mu\ell \triangle \pi_\mu\bar\ell| \leq 4d_\Haus(\ell,\bar\ell)$.
\end{proof}

\begin{proof}[of Proposition~\ref{prop:d_Haus_cont}]
Suppose $\ell,\bar\ell$ do not have the same direction.
Then clearly
\begin{equ}
|A(\ell)-A(\bar\ell)|
\leq |A|_{\gr\alpha}(|\ell|^\alpha + |\bar\ell|^\alpha)\;,
\end{equ}
and the conclusion follows by Lemma~\ref{lem:length_Haus_bound}.
Suppose now $\ell,\bar\ell$ have the same direction.
By additivity of $A$, using the notation of Lemma~\ref{lem:overlap_Haus_bound}, we have
\begin{equ}
|A(\ell)-A(\bar\ell)|
\leq |A|_{\alpha;\rho} (|X|^{\alpha/2} + |Y|^{\alpha/2})d(\ell,\bar\ell)^{\alpha/2} + |A|_{\gr\alpha}(|U|^\alpha + |V|^\alpha)\;,
\end{equ}
and the conclusion follows from Lemma~\ref{lem:overlap_Haus_bound}.
\end{proof}

For completeness, we record two further lemmas the proofs of which are obvious.

\begin{lemma}[Lower semi-continuity]\label{lem:lower_sem_cont}
Let $(A_n)_{n \geq 1}$ be a sequence of $E$-valued functions on $\mcX$ such that $\lim_{n \to \infty} A_n(\ell) = A(\ell)$ for every $\ell \in \mcX$.
Then for all $\alpha\in [0,1]$
\[
|A|_{\alpha;\rho} \leq \liminf_{n\to\infty} |A_n|_{\alpha;\rho}\quad \textnormal{ and } \quad |A|_{\gr\alpha} \leq \liminf_{n\to\infty} |A_n|_{\gr\alpha}\;.
\]
\end{lemma}


\begin{lemma}[Interpolation]\label{lem:interpolate}
For $0 \leq \bar\alpha\leq \alpha \leq 1$ and a function $A : \mcX \to E$, it holds that
\[
|A|_{\bar\alpha;\rho} \leq |A|_{0;\rho}^{1-\bar\alpha/\alpha} |A|_{\alpha;\rho}^{\bar\alpha/\alpha}\;,
\]
and similarly for $|\cdot|_{\gr\alpha}$.
\end{lemma}


\subsection{Additive functions from \texorpdfstring{$1$}{1}-forms}\label{subsec:additive_funcs_one_forms}

Let $\Omega^1$\label{Omega1 page ref} denote the space of all bounded, measurable $E$-valued one forms, i.e., all $A=\sum_{\mu=1}^d A_\mu \mrd x_\mu$ for which $A_\mu :\T^d \to E$ is bounded and measurable.

Consider $A\in\Omega^1$.
For $\gamma \in \mcC^{\var1}([s,t], \T^d)$, let us define
\begin{equ}
A(\gamma) \eqdef \int_s^t A(\gamma(u)) \mrd\gamma(u) = \sum_{\mu=1}^d \int_s^t A_\mu(\gamma(u)) \gamma_\mu'(u) \mrd u \in E\;.
\end{equ}
For $\ell \in \mcX$ with a parametrisation $\gamma \in \mcC^{\var1}([0,1],\T^d)$, we then define $A(\ell) \eqdef A(\gamma)$ (which is independent of the choice of parametrisation $\gamma$).
In such a way, we treat every element of $\Omega^1$ as an element of $\Omega$.

Note that this identification does not respect almost everywhere equality, i.e., if $A=\bar A$ a.e. on $\T^d$, it does not necessarily hold that $A(\ell)=\bar A(\ell)$ for all $\ell \in \mcX$.
However, we have the following.

\begin{proposition}\label{prop:zero_ae}
Let $A \in \Omega^1$.
If $A(\ell) = 0$ for all $\ell \in \mcX$, then $A$ is a.e. zero.
Conversely, suppose $A \in \Omega^1$ is a.e. zero and that $\ell \in \mcX$ is a continuity point of $A$ (as a function on $\mcX$). Then $A(\ell) = 0$.
\end{proposition} 

\begin{proof}
Let $\psi \in \mcC(\T^d,\R)$ and $\mu\in[d]$, and write
\begin{equ}
\scal{A_\mu,\psi} = \int_{z_\mu=0} \int_0^1 Y^z(t) \mrd X^z(t) \mrd z\;,
\end{equ}
where $Y^z\in\mcC([0,1],\R)$ is given by $Y^z(t) \eqdef \psi(z+te_\mu)$, and $X^z \in \mcC^{\var 1}([0,1],E)$ is given by $X^z(t) \eqdef \int_0^t A_\mu(z+s e_\mu) \mrd s$.
The first claim follows by noting that $X^z(t)$ is the evaluation of $A$ at an element of $\mcX$.
For the second claim, write $\ell=\{x+te_\mu \mid t \in [0,\lambda]\}$ for some $\lambda \geq 0$.
Let $(\phi_\eps)_{\eps > 0}$ be a smooth approximation of the Dirac delta $\delta_x$.
Denote $\ell_y\eqdef\{y+te_\mu \mid t \in [0,\lambda]\}$ and consider
\begin{equ}\label{eq:A_phi_test}
\int_{\T^d} A(\ell_y)\phi_\eps(y) \mrd y = \scal{A_\mu,\bar\phi_\eps}\;,
\end{equ}
where $\bar\phi_\eps(z) \eqdef \int_{0}^\lambda \phi_\eps(z-te_\mu)\mrd t$.
On the one hand, since $A_\mu$ is zero a.e., $\scal{A_\mu,\bar\phi_\eps} = 0$ for all $\eps > 0$.
On the other hand, $A(\ell_y) \to A(\ell)$ as $y \to x$ since $\ell$ is a continuity point of $A$, so that the LHS of~\eqref{eq:A_phi_test} converges to $A(\ell)$ as $\eps\to 0$, from which it follows that $A(\ell) =0$.
\end{proof}

As a consequence we may realise the space
\[
\mathring\Omega^1_0 \eqdef \{A \in \Omega^1 \ssep A \textnormal{ is continuous as a function on } \mcX\}
\]
simultaneously as a subspace of $\mcC(\mcX,E)$ and as a space of $E$-valued $L^\infty$ $1$-forms.
Note that, by Proposition~\ref{prop:d_Haus_cont}, every $A \in \Omega^1$ with $|A|_\alpha < \infty$ for some $\alpha > 0$ is in $\mathring\Omega^1_0$.

\begin{definition}
Let $\Omega^1_0$\label{page ref Omega 1 0} denote the closure of $\mathring\Omega^1_0$ in $\mcC(\mcX,E)$ under the uniform norm.
For $\alpha \in (0,1]$, let $\Omega^1_\alpha$\label{page ref Omega 1 alpha} denote the closure of $\{A \in \Omega^1_0 \ssep |A|_\alpha < \infty\}$ in $\Omega_\alpha$ equipped with the norm $|\cdot|_\alpha$.
\end{definition}

\subsection{Embeddings}

In this subsection, we show that $\Omega_\alpha$ is compactly embedded in $\Omega^1_{\bar \alpha}$ for $\bar \alpha < \alpha$, and that the latter is continuously embedded in $\Omega^1_{\mcC^{\bar\alpha-1}}$, the H{\"o}lder--Besov space of distributions commonly used in anaysis of SPDEs~\cite{Hairer14, GIP15}.

\subsubsection{Dyadic approximations and compact embeddings}

Fix in this section $A \in \Omega$.
We suppose further that $A(\ell) = 0$ unless $\ell$ has direction $\mu\in[d]$. 
We construct a sequence of functions $A^{(N)} \in \Omega^1_0$ (which serve as dyadic approximations to $A$) as follows.
For $x \in \T^d$ and $N \geq 0$, let $k$ be the unique integer in $\{0,\ldots,2^N-1\}$ such that $\pi_\mu x \in [k2^{-N},(k+1)2^{-N})$.
Let $\ell_x^{(N)}$ be the unique axis line segment of length $2^{-N}$ containing $x$ such that $\pi_\mu \ell_x^{(N)} = [k2^{-N},(k+1)2^{-N}]$.
We then define $A^{(N)}_\mu(x) \eqdef 2^{N}A(\ell_x^{(N)})$ and $A^{(N)} = A^{(N)}_\mu \mrd x_\mu$.

\begin{lemma}\label{lem:upper_bound_dyadics}
Let $\alpha \in [0,1]$.
It holds that
\begin{equ}
|A^{(N)}|_{\gr\alpha} \leq 3^{1-\alpha}|A|_{\gr\alpha}\;,\quad |A^{(N)}|_{\alpha;\rho} \leq 3^{1-\alpha/2}|A|_{\alpha;\rho}\;.
\end{equ}
\end{lemma}

\begin{proof}
For the first inequality, let us write $\ell\in\mcX$ as $\ell=\ell_1\cup\ell_2\ldots\cup\ell_n$, where $\ell_i$ and $\ell_{i+1}$ are joinable for $i \in \{1,\ldots, n-1\}$, and each $\ell_i$ is contained in a single cell, i.e., a set of the form $\pi_\mu^{-1}[k2^{-N},(k+1)2^{-1}]$.
Then
\begin{equs}
|A^{(N)}(\ell)|
&\leq |A^{(N)}(\ell_1)| + |A^{(N)}(\ell_n)| + \Big|\sum_{i=2}^{n-1} A^{(N)}(\ell_i) \Big|
\\
&\leq (|\ell_1|+|\ell_n|)2^{N}\Big(\sup_{|\bar\ell|=2^{-N}}|A(\bar\ell)| \Big) + \Big|\sum_{i=2}^{n-1} A(\ell_i) \Big|
\\
&\leq |A|_{\gr\alpha}\Big[ |\ell_1|^\alpha+|\ell_n|^\alpha + \Big(\sum_{i=2}^n |\ell_i|\Big)^\alpha \Big]
\\
&\leq |A|_{\gr\alpha}3^{1-\alpha}|\ell|^\alpha\;.
\end{equs}
For the second inequality, let $\ell,\bar\ell \in \mcX$ be parallel.
Let us decompose $\ell,\bar\ell$ exactly as above.
Observe that
\begin{equs}
|A^{(N)}(\ell_1)-A^{(N)}(\bar\ell_1)|
&\leq |\ell_1|2^{N}\sup_{a,b} |A(a)-A(b)|
\\
&\leq |\ell_1|2^{N} |A|_{\alpha;\rho} d(\ell,\bar\ell)^{\alpha/2}2^{-N\alpha/2}
\\
&\leq |A|_{\alpha;\rho}|\ell_1|^{\alpha/2}d(\ell,\bar\ell)^{\alpha/2}\;,
\end{equs}
where the first supremum is taken over all parallel $a,b \in \mcX$ which are in the same cell and for which $d(a,b) = d(\ell,\bar\ell)$.
The same holds for $|A^{(N)}(\ell_n)-A^{(N)}(\bar\ell_n)|$.
For the middle part, we simply have
\begin{equ}
\Big|\sum_{i=2}^{n-1} A^{(N)}(\ell_i) - A^{(N)}(\bar\ell_i) \Big|
\leq |A|_{\alpha;\rho}\Big(\sum_{i=2}^{n-1}|\ell_i|\Big)^{\alpha/2}d(\ell,\bar\ell)^{\alpha/2}\;.
\end{equ}
It follows that
\begin{equ}
|A^{(N)}(\ell) - A^{(N)}(\bar\ell)| \leq |A|_{\alpha;\rho}3^{1-\alpha/2}d(\ell,\bar\ell)^{\alpha/2}|\ell|^{\alpha/2}
\end{equ}
\end{proof}

\begin{lemma}\label{lem:conv_unif}
Suppose $A$ is continuous as a function on $\mcX$.
Then
\[
\lim_{N\to \infty} \sup_{\ell \in\mcX}|A^{(N)}(\ell) - A(\ell)| = 0\;.
\]
\end{lemma}

\begin{proof}
Since $A(\ell)=0$ for all $\ell \in \mcX$ consisting of a single point, (uniform) continuity of $A$ on $\mcX$ implies $\lim_{\eps\to 0} \sup_{|\ell| \leq \eps}|A(\ell)| = 0$.
The conclusion follows by additivity and the definition of $A^{(N)}$.
\end{proof}

\begin{lemma}\label{lem:alpha_baralpha_compact}
For $0 \leq \bar\alpha<\alpha \leq 1$, the unit ball of $\Omega_{\alpha}$ is compact in $\Omega_{\bar\alpha}$.
\end{lemma}

\begin{proof}
Proposition~\ref{prop:d_Haus_cont} implies that $\frac\alpha2$-H{\"o}lder norm of $A\in\Omega_\alpha$ is bounded by $\lesssim |A|_{\alpha}$, hence the unit ball of $\Omega_{\alpha}$ is equicontinuous and bounded in $\mcC(\mcX,E)$.
Since $\mcX$ is compact, the claim follows by Arzel{\`a}--Ascoli and Lemmas~\ref{lem:lower_sem_cont} and~\ref{lem:interpolate}.
\end{proof}

Combining Lemmas~\ref{lem:interpolate},~\ref{lem:upper_bound_dyadics},~\ref{lem:conv_unif}, and~\ref{lem:alpha_baralpha_compact}, we obtain the following.

\begin{proposition}\label{prop:omega_omega1_inclusion}
For $0 \leq \bar\alpha<\alpha \leq 1$, $\Omega_{\alpha}$ is compactly embedded in $\Omega^{1}_{\bar\alpha}$.
\end{proposition}

\subsubsection{H{\"o}lder--Besov spaces}

For $\alpha < 0$, we recall the space of distributions $\mcC^\alpha(\T^d,E)$ from~\cite{Hairer14}.
For a function $\psi : \R^d \to \R$ and $\lambda \in (0,1]$, we denote $\psi^\lambda(y) \eqdef \lambda^{-d}\psi(\lambda^{-1}y)$.
Let $r \eqdef -\floor\alpha$ and $\mcB^r$ denote the set of all smooth functions $\psi \in \mcC^\infty(\R^d)$ with $|\psi|_{\mcC^r} \leq 1$ and $\supp \psi \subset B(0,1/4)$.
Here $\mcC^r$ is the space of $(r-1)$-times differentiable functions with Lipschitz $(r-1)$-th derivative, see~\cite[Sec.~2.2]{Hairer14}.
For $\lambda \in (0,1]$, since $\supp \psi^\lambda$ has diameter at most $1/2$, we can canonically treat $\psi^\lambda$ as an element of $\mcC^\infty(\T^d)$.
For $x\in\T^d$, we then define $\psi^\lambda_x(y) = \psi^\lambda(y-x)$.
Finally, we let  $(\mcC^\alpha(\T^d,E),|\cdot|_{\mcC^\alpha})$ denote the space of distributions $\xi \in \mcD'(\T^d,E)$ for which
\[
|\xi|_{\mcC^\alpha} \eqdef \sup_{\lambda\in(0,1]}\sup_{\psi \in \mcB^r} \sup_{x \in \T^d} \frac{|\scal{\xi,\psi^\lambda_x}|}{\lambda^\alpha} < \infty\;.
\]
We analogously define $(\Omega^1_{\mcC^\alpha},|\cdot|_{\mcC^\alpha})$ as the space of distributional $1$-forms $A=\sum_{\mu=1}^d A_\mu \mrd x_\mu$ for which $|A|_{\mcC^\alpha} \eqdef \sum_{\mu=1}^d |A_\mu|_{\mcC^\alpha} < \infty$.
For $\alpha=0$, we define $(\Omega^1_{\mcC^0},|\cdot|_{\mcC^0})$ as the space of $E$-valued $L^\infty$ $1$-forms, i.e., $A=\sum_{\mu=1}^d A_\mu \mrd x_\mu$ with $|A|_{\mcC^0} \eqdef \sum_{\mu=1}^d |A_\mu|_{L^\infty} < \infty$.

Recall the definition of $\mathring\Omega^1_0$ from Section~\ref{subsec:additive_funcs_one_forms}, which, by Proposition~\ref{prop:zero_ae}, is a subspace of $\Omega^1_{\mcC^0}$.

\begin{proposition}\label{prop:embedding}
For $\alpha \in (0,1]$, the space $(\mathring\Omega^1_0, |\cdot|_{\gr\alpha})$ is continuously embedded in $(\Omega^1_{\mcC^{\alpha-1}}, |\cdot|_{\mcC^{\alpha-1}})$.
\end{proposition}

\begin{lemma}\label{lem:test_upper_bound}
Let $A \in \Omega^1$, $\alpha \in (0,1]$, $\psi \in \mcC(\T^d)$, $\lambda \in (0,1]$, $\mu \in [d]$, and $z \in \T^d$.
Then
\begin{equ}
\Big| \int_{0}^1 A_\mu(z+\lambda te_\mu) \psi(z+\lambda t e_\mu) \mrd t \Big| \lesssim |\psi|_{\mcC^1} \lambda^{\alpha}|A|_{\gr\alpha}\;,
\end{equ}
with a proportionality constant depending only on $\alpha$.
For $\alpha=1$, we furthermore have
\begin{equ}
\Big| \int_{0}^1 A_\mu(z+\lambda te_\mu) \psi(z+\lambda t e_\mu) \mrd t \Big| \leq |\psi|_{L^\infty}|A|_{\gr1}\;.
\end{equ}
\end{lemma}
\begin{proof}
Define the paths $X \in \mcC^{\Hol1}([0,1], E)$ by $X(t) = \int_0^t \lambda A_\mu(z+\lambda s e_\mu) \mrd s$ and $Y \in \mcC([0,1], \R)$ by $Y(t) = \psi(z+\lambda t e_\mu)$.
Observe that $X = \ell_A$ for $\ell \in \mcX$ defined by $\ell \eqdef \{z+\lambda t e_\mu \ssep t \in [0,1]\}$.
Since $|\ell|=\lambda$, it follows by the Lemma~\ref{lem:ellA_Hol_bound} that $|X|_{\Hol{\alpha}} \leq \lambda^\alpha |A|_{\gr\alpha}$.
Furthermore we have $|Y|_{\Hol{1}} \leq \lambda|\psi|_{\mcC^1}$ and $|Y|_{L^\infty} \leq |\psi|_{L^\infty}$.
Then by Young integration
\begin{equ}
\Big| \int_0^1 Y(t) \mrd X(t) \Big| \lesssim |Y|_{\Hol{1}}|X|_{\Hol\alpha} \lesssim |\psi|_{\mcC^1}\lambda^{\alpha+1}|A|_{\gr\alpha}\;,
\end{equ}
and by classical Riemann--Stieltjes integration
\begin{equ}
\Big| \int_0^1 Y(t) \mrd X(t) \Big| \leq |Y|_{L^\infty}|X|_{\Hol1} \leq |\psi|_{L^\infty}\lambda|A|_{\gr1}\;.
\end{equ}
\end{proof}

\begin{proof}[of Proposition~\ref{prop:embedding}]
Let $A\in\Omega^1$, $\mu \in [d]$ and $\psi \in \mcC^1(\R^d)$ be a test function with $\supp \psi \subset B(0,1/4)$.
Let $\bar B(0,\lambda/4)$ be the ball of radius $\lambda/4$ centered at $0$ in the hyperplane $x_\mu=0$.
Then uniformly in $\lambda \in (0,1]$
\begin{equs}
|\scal{A_\mu,\psi^\lambda_0}| &= \Big|\int_{B(0,\lambda/4)} A_\mu(y) \psi_0^\lambda(y) \mrd y \Big|
\\
&\leq \int_{\bar B(0,\lambda)} \mrd z \Big|\int_{-1/4}^{1/4} \mrd t \lambda A(z+\lambda te_\mu)\psi^\lambda_0(z+\lambda te_\mu) \Big|
\\
&\lesssim |\psi|_{\mcC^1} \lambda^{\alpha-1}|A|_{\gr\alpha}\;,
\end{equs}
where the final inequality follows from Lemma~\ref{lem:test_upper_bound} and the facts that $|\psi^\lambda_0|_{\mcC^1} = \lambda^{-d-1}|\psi|_{\mcC^1}$ and $\bar B(0,\lambda) \sim \lambda^{d-1}$.
This proves the desired result for $\alpha\in(0,1)$.
For $\alpha=1$, we have in the same way from Lemma~\ref{lem:test_upper_bound}
\[
|\scal{A_\mu,\psi^\lambda_0}| \leq |\psi|_{L^\infty}|A|_{\gr 1}\;,
\]
where we now used $|\psi^\lambda_0|_{L^\infty} = \lambda^{-d}|\psi|_{L^\infty}$.
It readily follows that $|A|_{\mcC^0} \lesssim |A|_{\gr 1}$ as desired (where we have used that $L^\infty$ is the dual of $L^1$).
\end{proof}

\begin{corollary}\label{cor:cont_embed}
For $\alpha \in (0,1]$, $\Omega^1_\alpha$ is continuously embedded in $\Omega^1_{\mcC^{\alpha-1}}$.
\end{corollary}

\subsection{Lattice approximations}

We will see in the following sections that lattice gauge theory provides us with random approximations of elements in $\Omega_{\alpha}$ defined on lattices.
We show that one can take projective weak limit points of these random variables in $\Omega_\alpha$.
Recall the definition of $\Omega^{1,(N)}$ and note that every $A\in\Omega$ canonically defines an element of $\Omega^{1,(N)}$.

\begin{definition}
Let $\mcX^{(N)}$\label{page ref mcX N} denote the subset of all $\ell \in \mcX$ which are the union of bonds in $\overline\Bonds_N$.
For $A \in \Omega^{1,(N)}$, let $|A|^{(N)}_{\alpha;\rho}$, $|A|^{(N)}_{\gr\alpha}$, and $|A|^{(N)}_{\alpha}$ be defined as in Definition~\ref{def:rhoHol} but with the restriction $\ell,\bar\ell \in \mcX^{(N)}$.
\end{definition}

\begin{lemma}\label{lem:lattice_bound}
For any continuous $A\in\Omega$ and $\alpha \in [0,1]$, it holds that
\[
|A|_{\alpha;\rho} = \lim_{N \to \infty}|A|^{(N)}_{\alpha;\rho}\quad\textnormal{ and } \quad |A|_{\gr\alpha;} = \lim_{N \to \infty}|A|^{(N)}_{\gr\alpha}\;.
\]
\end{lemma}

\begin{proof}
Let $\ell,\bar\ell \in \mcX$ be parallel.
Observe that there exist sequences $\ell_N,\bar\ell_N\in\mcX^{(N)}$ such that $\ell_N$ and $\bar\ell_N$ are parallel for each $N$, and $\lim_{N\to\infty}d_\Haus(\ell,\ell_N) = \lim_{N\to\infty}d_\Haus(\bar\ell,\bar\ell_N) = 0$.
By the assumption that $A$ is continuous, we have  $A(\ell) = \lim_{N\to\infty}A(\ell_N)$ and likewise for $\bar\ell$.
Furthermore, clearly $\lim_{N\to\infty}\rho(\ell_N,\bar\ell_N) = \rho(\ell,\bar\ell)$.
Both equalities readily follow.
\end{proof}

\begin{theorem}\label{thm:proj_limit}
Let $0 \leq \bar\alpha < \alpha \leq 1$.
Suppose that for every $N \geq 0$, $A^{(N)}$ is an $\Omega^{1,(N)}$-valued random variable such that $(|A^{(N)}|^{(N)}_{\alpha})_{N = 0}^\infty$ is a tight family of real random variables.
Then there exists a subsequence $(N_k)_{k=0}^\infty$ and a $\Omega_{\alpha}$-valued random variable $A$ such that
\begin{equ}\label{eq:conv_law}
\textnormal{$A^{(N_k)} \to A$ in law as $\Omega^{1,(M)}$-valued random variables for all $M \geq 0$}\;,
\end{equ}
and for all $K \geq 0$
\begin{equ}\label{eq:prob_bound}
\P\big[|A|_\alpha > K\big] \leq \limsup_{k \to \infty} \P\big[|A^{(N_k)}|^{(N_k)}_\alpha > K\big]\;.
\end{equ}
\end{theorem}

\begin{proof}
By tightness and lower semi-continuity, for every $M \geq 0$ there exists a subsequence $N_k$ and an $\Omega^{1,(M)}$-valued random variable $\tilde A^{(M)}$ such that $A^{(N_k)} \to \tilde A^{(M)}$ in law as $\Omega^{1,(M)}$-valued random variables.
By a diagonalisation argument, we may suppose that the same subsequence $N_k$ works for all $M \geq 0$.
In particular, for all $\bar M \geq M \geq 0$, $\tilde A^{(\bar M)}$ and $\tilde A^{(M)}$ have the same law as $\Omega^{1,(M)}$-valued random variables.
For any $\bar\alpha \in (0,\alpha)$, it follows from the existence of projective limits of measures~\cite[Thm.~9.12.1]{Bogachev07} and Proposition~\ref{prop:omega_omega1_inclusion} that there exists an $\Omega^{1}_{\bar\alpha}$-valued random variable $A$ for which~\eqref{eq:conv_law} holds (we used here that $\Omega^{1}_{\bar\alpha}$ is Polish).
The bound~\eqref{eq:prob_bound} (and thus the fact that $A$ a.s. takes values in $\Omega_\alpha$) follows from Lemma~\ref{lem:lattice_bound}.
\end{proof}

\section{Deterministic bounds}
\label{sec:det_bounds}

In this section we collect the necessary deterministic results concerning lattice gauge theory.
We restrict henceforth to the case $\T^d = \T^2$.
We emphasise however that this assumption is not necessary in this section, and a similar analysis can be performed in arbitrary dimension.
The presentation however does simplify significantly in this case, and furthermore the probabilistic bounds in the following section depend crucially on the fact that $d=2$.

We will henceforth take $E=\mfg$ when considering the spaces $\Omega^{1,(N)}(\T^2,\mfg)$.
Throughout this section let $N_1 \geq 0$ and $U \in \mfA^{(N_1)}$.

\begin{definition}\label{def:development}
For $N \leq N_1$ and a rectangle $r\subset\Lambda_N$, let $p_1,\ldots, p_k$ denote the plaquettes of $\Lambda_N$ ordered so that neither $p_1-2^{-N}e_1$ nor $p_1-2^{-N}e_2$ are contained in $r$ and so that the boundaries of $p_{i+1}$ and $p_i$ share a common bond for $i=1,\ldots, k-1$ (note this defines the order uniquely).
Let $r_i$ denote the subrectangle of $r$ consisting of the plaquettes $p_1,\ldots, p_i$.
See Figure~\ref{fig:rect} for an example. 
We call the anti-development of $U$ along $r$ the $\mfg$-valued sequence $(X_i)_{i=0}^k$ with $X_0 = 0$ and increments $X_i - X_{i-1} \eqdef \log(U(\partial r_{i-1})^{-1}U(\partial r_i))$.
\end{definition}

\begin{figure}[t]
\centering
\begin{tikzpicture}[scale = 2.0]

\draw[thick] (0.5,0.5)--++(3.0,0)--++(0,0.5)--++(-3.0,0)--++(0,-0.5);

\foreach \x in {0,...,4}{
	\draw[thick,densely dotted] (1+\x/2,0.5)--++(0,0.5);
}
\foreach \x in {1,...,6}{
	\node (p\x) at (1/4 + \x/2,3/4) {$p_{\x}$};
}

\draw[decorate, decoration={brace}, yshift=0.5ex]  (0.5,1) -- node[above=0.4ex] {$r_4$} (2.5,1);
\draw[decorate, decoration={brace,mirror}, yshift=-0.5ex]  (0.5,0.5) -- node[below=0.4ex] {$r=r_6$}  (3.5,0.5);

\foreach \x in {0,...,8}{
  \foreach \y in {0,...,3}{
    \fill[black] (0+\x/2,0+\y/2) circle (0.02);
  }
}
\end{tikzpicture}
\caption{
Example of $r$ with $6$ plaquettes.
}
\label{fig:rect}
\end{figure}
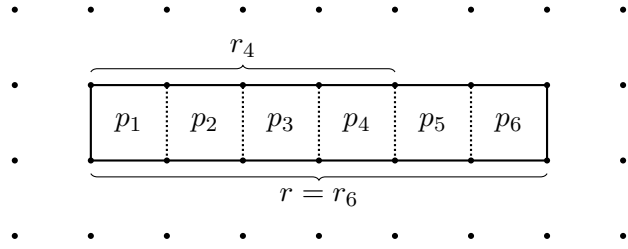

For an integer $N \leq N_1$ and a rectangle $r \subset \Lambda_N$, consider the conditions for some $\bar C \geq 0$ and $\alpha \in \R$
\begin{equ}\label{eq:hol_bound_rectangle}
|\log U(\partial r)| \leq \bar C |r|^{\alpha/2}\;,
\end{equ}
and for some $q \geq 1$
\begin{equ}\label{eq:hol_bound_antidev}
|X|_{\var{q}} \leq \bar C |r|^{\alpha/2}\;,
\end{equ}
where $(X_i)_{i=1}^k$ is the anti-development of $U$ along $r$.

\begin{remark}
If $r$ is a single plaquette, then $|X|_{\var q}$ does not depend on $q$ and~\eqref{eq:hol_bound_antidev} is equivalent to~\eqref{eq:hol_bound_rectangle}.
\end{remark}

\begin{remark}
If $g \in \mfG^{(N)}$, then $U^g(\partial r_i) = g(z)U(\partial r_i)g(z)^{-1}$ where $z \in \Lambda_N$ is the origin of $r$.
Hence $|\log U(\partial r)|$ and $|X|_{\var{q}}$ are both gauge invariant.
\end{remark}

\begin{remark}
As the name suggests, the development of $X$ into $G$ is exactly the sequence $(U(\partial r_i))_{i=1}^k$.
As a result, by Young integration, if~\eqref{eq:hol_bound_antidev} holds for some $q < 2$, then so does~\eqref{eq:hol_bound_rectangle} (potentially with a larger $\bar C$).
In our situation, we will only have~\eqref{eq:hol_bound_antidev} for $q > 2$, in which case~\eqref{eq:hol_bound_rectangle} would only be implied by~\eqref{eq:hol_bound_antidev} if $X$ is replaced by its rough path lift (and our probabilistic estimates in the following section indeed imply this stronger bound).
However we choose the current formulation to keep the assumptions in this section more elementary and since the bound~\eqref{eq:hol_bound_antidev} will only be used in the ``Young regime'', cf. Lemma~\ref{lem:hol_sum_bound}.
\end{remark}

The main result of this section can be stated as follows.

\begin{theorem}\label{thm:Landau_axial}
Suppose there exist $\alpha \in (\frac23,1)$, $\bar C \geq 0$, and $q\in[1,\frac{1}{1-\alpha})$, such that for all integers $N \leq N_1$ we have
\begin{enumerate}[label=(\roman*)]
\item \eqref{eq:hol_bound_antidev} for all rectangles $r \subset \Lambda_N$, and
\item\label{pt:axial} \eqref{eq:hol_bound_rectangle} for all rectangles of the form $r=((0,n2^{-N}), m2^{-N}e_1, 2^{-N}e_2)\subset\Lambda_N$ where $1 \leq m < 2^N$ and $0 \leq n < 2^N$ (see Figure~\ref{fig:axial} for an example).
\end{enumerate}
Suppose further that $G$ is simply connected.
Then there exists $A \in \Omega^{1,(N_1)}$ such that $\exp A = U^g$ for some $g\in \mfG^{(N_1)}$ and for every $\bar\alpha < \alpha$, there exists $C\geq 0$, independent of $N_1$, such that $|A|^{(N_1)}_{\bar\alpha} \leq C$.
\end{theorem}

\begin{proof}
By Proposition~\ref{prop:axial} we can apply the axial gauge for sufficiently large $N_0 \geq 1$ until the assumptions of Theorem~\ref{thm:Landau} are satisfied, after which we can apply the binary Landau gauge for $N_0 \leq N \leq N_1$.
\end{proof}

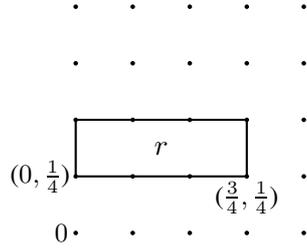
\begin{figure}[t]
\centering
\begin{tikzpicture}[scale = 1.5]

\foreach \x in {0,...,4}{
  \foreach \y in {0,...,4}{
    \fill[black] (0+\x/2,0+\y/2) circle (0.02);
  }
}
\node (0) at (-1/8,0) {\small{$0$}};
\node (y) at (-2.5/8,1/2) {\small{$(0,\frac14)$}};
\node (x) at (1.5,2.5/8) {\small{$(\frac34,\frac14)$}};
\node (x) at (3/4,3/4) {$r$};

\draw[thick] (0,0.5)--++(1.5,0)--++(0,0.5)--++(-1.5,0)--++(0,-0.5);
\end{tikzpicture}
\caption{
Example of $r$ from item~\ref{pt:axial} of Theorem~\ref{thm:Landau_axial} with $N=2$, $n=1$, $m=3$.
}
\label{fig:axial}
\end{figure}

\subsection{Binary Landau gauge}\label{subsec:Landau_gauge}

Throughout this subsection, let us fix $N_0 \leq N_1$.
We should think of $N_0$ as providing a fixed medium scale while we take $N_1 \to \infty$.
We will define $A \in \Omega^{1,(N_1)}$ and $g\in\mfG^{(N_1)}$ such that $\exp(A) = U^g$ with explicit bounds on $|A|_{\bar\alpha}^{(N_1)}$.

\begin{remark}\label{rem:poisson}
We will be guided by the following observation.
Let $A$ be a smooth $\mfg$-valued $1$-form on a closed hypercube $B$ in $\R^d$ with curvature $F_{\mu\nu} = \partial_\mu A_\nu - \partial_\nu A_\mu + [A_\mu,A_\nu]$.
Suppose that $A$ satisfies the Landau gauge $\sum_{\mu=1}^d \partial_\mu A_\mu = 0$ in the interior of $B$.
For $\mu=1,\ldots, d$, let $\partial_\mu B$ denote the two hyperplanes on the boundary of $B$ perpendicular to $e_\mu$.
Suppose further that $A$ satisfies the $(d-1)$-dimensional Landau gauge on $\partial_\mu B$, i.e, $\sum_{\nu\neq \mu}\partial_\nu A_\nu = 0$.
Combined with the $d$-dimensional Landau gauge, we obtain the Neumann boundary condition $\partial_\mu A_\mu \lvert_{\partial_\mu B} = 0$.
To recover $A_\mu$ from $F$, we suppose that $A_\mu$ has a prescribed boundary condition on $\partial_\nu B$ for $\nu\neq \mu$, and observe that in the interior of $B$
\begin{equs}
\sum_{\nu=1}^d \partial_\nu F_{\mu\nu} &= \sum_{\nu=1}^d  \partial_{\nu\mu} A_\nu - \partial_{\nu\nu} A_\mu + [\partial_\nu A_\mu,A_\nu] + [A_\mu,\partial_\nu A_\nu]
\\
&= \Delta A_\mu + \sum_{\nu=1}^d [\partial_\nu A_\mu,A_\nu]\;.
\end{equs}
If $A$ is small or if $G$ is Abelian, the final terms can be ignored and we are left with a Poisson equation for $A_\mu$ with a mixed Dirichlet--Neumann boundary condition (we ignore the non-smoothness of $\partial B$ in this discussion).
The probabilistic representation of the solution is $A_\mu(x) = \E[A_\mu(W_\tau) + \int_0^\tau \sum_\nu \partial_\nu F_{\mu\nu}(W_s)\mrd s]$, where $W$ is a Brownian motion started at $x$, conditioned to exit $B$ at $\partial B \setminus \partial_\mu B$, and $\tau$ is the first exit time of $W$ from $B$.
Using this representation (or the classical maximum principle) we see that $A_\mu$ is bounded by its value on $\partial B\setminus \partial_\mu B$ plus contributions from $\partial_\nu F_{\mu\nu}$.

Provided the contribution from $\partial_\nu F_{\mu\nu}$ is small, this allows us to bound $A$ on smaller scales by its value on large scales.
The procedure in this subsection can be seen as a discrete version of this boundary value problem with a random walk approximation.
\end{remark}

We define $A$ and $g$ inductively.
To start, let $N=N_0$ and $A(\alpha) \eqdef \log U(\alpha)$ for every bond $\alpha \in \overline\Bonds_{N_0}$.
Correspondingly, $g(x) = 1_G$ for all $x \in \Lambda_{N_0}$.

Suppose we have defined $A$ and $g$ on $\Bonds_{N-1}$ and $\Lambda_{N-1}$ respectively for $N_0 < N \leq N_1$.
To extend the definition to $N$, we consider intermediate lattices
\begin{equs}
\Lambda^0_N = \Lambda_{N-1} \subset \Lambda^1_N \subset \Lambda^2_N = \Lambda_{N}
\end{equs}
where $\Lambda_N^k$ is the subset of $\Lambda_{N}$ consisting of vertices $x=(x_1,x_2)$ for which at most $k$ coordinates have binary power at most $N$ (see Section~\ref{subsec:one-forms} for the definition of binary power).
We correspondingly define the set of bonds $\Bonds_{N}^{k}$ by $\Bonds_N^0 = \Bonds_{N-1}$ and for $k=1,2$ as the set of ordered pairs $(x,y)$ where $x,y \in \Lambda_N^{k}$ with $|x-y| = 2^{-N}$ (in particular $\Bonds_N^2 = \Bonds_N$).

For $k = 1,2$, we define $A$ and $g$ on $\Bonds_{N}^k$ and $\Lambda_N^k$ as follows.
Let $x=(x_1,x_2)$ be a site of $\Lambda_{N}^k$ for which $x_{\mu_1},\ldots,x_{\mu_k}$ have binary power $N$ (so that $x$ is not a site of $\Lambda^{k-1}_{N-1}$).
We introduce the shorthand $x^\pm_\mu\eqdef x\pm2^{-N}e_\mu$.

If $k=1$, we define
\[
A(x,x^+_{\mu_1}) = A(x^-_{\mu_1},x) \eqdef \frac12 A(x^-_{\mu_1},x^+_{\mu_1})\;.
\]
We then extend the definition of $g$ to $x$ by enforcing
\begin{equ}
\exp A(x,x^+_{\mu_1}) = g(x) U(x,x^+_{\mu_1})g(x^+_{\mu_1})^{-1}\;.
\end{equ}
It clearly holds that $\exp A = U^g$ on $\Bonds_N^1$ (with $U^g$ defined in the obvious way).

If $k=2$, let $p_1,p_2,p_3,p_4$ be the four plaquettes of $\Lambda_N$ one of whose corners is $x$, ordered from the positive quadrant anti-clockwise, see Figure~\ref{fig:plaquettes}.
Note that the origin of $p_i$ is a point $z_i\in\Lambda_{N-1}$ which is the corner of $p_i$ opposite to $x$.
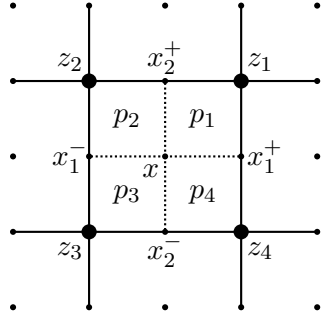
\begin{figure}[t]
\centering
\begin{tikzpicture}[scale = 2.0]
\foreach \x in {0,1}{
  \foreach \y in {0,1}{
    \fill[black] (0.5+\x,0.5+\y) circle (0.05);
  }
}
\draw[thick] (0.5,0.5)--++(1,0)--++(0,1)--++(-1,0)--++(0,-1);
\draw[thick] (0.5,0.5)--++(0,-0.5);
\draw[thick] (0.5,0.5)--++(-0.5,0);

\draw[thick] (1.5,1.5)--++(0,0.5);
\draw[thick] (1.5,1.5)--++(0.5,0);

\draw[thick] (1.5,0.5)--++(0,-0.5);
\draw[thick] (1.5,0.5)--++(0.5,0);

\draw[thick] (0.5,1.5)--++(0,0.5);
\draw[thick] (0.5,1.5)--++(-0.5,0);

\draw[thick,densely dotted] (1,1)--++(0,0.5);
\draw[thick,densely dotted] (1,1)--++(0,-0.5);
\draw[thick,densely dotted] (1,1)--++(0.5,0);
\draw[thick,densely dotted] (1,1)--++(-0.5,0);
\node (p1) at (5/4,5/4) {$p_1$};
\node (p2) at (3/4,5/4) {$p_2$};
\node (p3) at (3/4,3/4) {$p_3$};
\node (p4) at (5/4,3/4) {$p_4$};

\node (x) at (1-3/32,1-3/32) {$x$};

\node (xm2) at (1,0.5-1/8) {$x^-_2$};
\node (xp2) at (1,1.5+5/32) {$x^+_2$};
\node (xm1) at (0.5-1/8,1) {$x^-_1$};
\node (xp1) at (1.5+5/32,1) {$x^+_1$};

\node (z1) at (1.5+4/32,1.5+4/32) {$z_1$};
\node (z2) at (0.5-4/32,1.5+4/32) {$z_2$};
\node (z3) at (0.5-4/32,0.5-4/32) {$z_3$};
\node (z4) at (1.5+4/32,0.5-4/32) {$z_4$};

\foreach \x in {0,...,4}{
  \foreach \y in {0,...,4}{
    \fill[black] (0+\x/2,0+\y/2) circle (0.02);
  }
}
\end{tikzpicture}
\caption{
Large circles are points of $\Lambda_{N-1}$, small circles are points of $\Lambda_N$.
}
\label{fig:plaquettes}
\end{figure}
Define
\begin{equs}
\partial_2 F_{12}(x) &\eqdef \log U^g(\partial p_1)-\log U^g(\partial p_4)\;,
\\
\partial_1 F_{21}(x) &\eqdef \log U^g(\partial p_2)-\log U^g(\partial p_1)\;,
\\
\partial_2 F_{12}(x_1^-) &\eqdef \log U^g(\partial p_2)-\log U^g(\partial p_3)\;,
\\
\partial_1 F_{21}(x_2^-) &\eqdef \log U^g(\partial p_3)-\log U^g(\partial p_4)\;.
\end{equs}
Note that $U^g(\partial p_i) = g(z_i)U(\partial p_i)g(z_i)^{-1}$, which well-defined since $g(z_i)$ is defined by induction.

\begin{lemma}\label{lem:CBH_bound}
For all $n \geq 1$, there exists $C>0$ depending only on $n$ and $G$, such that for all $A_1,\ldots,A_n\in \mfg$, it holds that
\begin{equ}
\Big| \log(e^{A_1}\ldots e^{A_n}) - \sum_{i=1}^n A_i \Big| \leq C \sum_{i=1}^n |A_i|^2\;.
\end{equ}
\end{lemma}

\begin{proof}
An immediate consequence of the compactness of $G$ and non-zero radius of convergence of the Campbell--Baker--Hausdorff formula.
\end{proof}

\begin{lemma}\label{lem:mcE_bound}
Let $A$ and $g$ be defined as above on $\Bonds_{N}^1$ and $\Lambda_N^1$ respectively.
For $x \in \Lambda_N^2$ as above, denote
\[
\delta \eqdef \sum_{i=1}^4|\log U(\partial p_i)|  + \sum_{\mu\neq\nu} |A^N_\mu(x^\pm_\nu)|\;.
\]
Then there exist $\mcE_i \in \mfg$ for $i=1,2,3$, a constant $C\geq 0$ depending only on $G$, and a unique choice for $g(x)$, such that $|\mcE_i| \leq C\delta^2$ and such that
\begin{equs}
A_1^N(x) &\eqdef \frac{A^N_1(x^+_2)+A^N_1(x^-_2)}{2}+ \frac{3}{8}\partial_2 F_{12}(x) + \frac{1}{8}\partial_2 F_{12}(x_1^-)\;,
\\
A_2^N(x) &\eqdef \frac{A^N_2(x^+_1) + A^N_2(x^-_1)}{2}+ \frac{3}{8}\partial_1 F_{21}(x) + \frac{1}{8}\partial_1 F_{21}(x^-_2) + \mcE_1 \;,
\\
A_1^N(x_1^-) &\eqdef \frac{A^N_1(x^+_2) +A^N_2(x^-_2) }{2} + \frac{3}{8}\partial_2 F_{12}(x^-_1) + \frac{1}{8}\partial_1F_{12}(x) + \mcE_2\;,
\\
A_2^N(x_2^-) &\eqdef \frac{A^N_2(x^+_1)+A^N_2(x^-_1)}{2}+ \frac{3}{8}\partial_1 F_{21}(x_2^-) + \frac{1}{8}\partial_1 F_{21}(x) + \mcE_3\;,
\end{equs}
satisfy $\exp A = U^g$ on $\Bonds_N \cap (p_1\cup\ldots \cup p_4)$.
\end{lemma}

\begin{remark}\label{rem:rand_walk}
Following Remark~\ref{rem:poisson}, the ratios $\frac38$ and $\frac18$ arise from the following observation:
let $X$ be a random walk on the bonds of $p_1,\ldots, p_4$ parallel to $e_1$ starting on $(x,x+2^{-N}e_1)$ which is stopped the first time it hits the boundary of $p_1\cup \ldots \cup p_4$.
Then $X$ will stop on $\partial (p_1\cup p_4)$ with probability $\frac34$ and on $\partial (p_2\cup p_3)$ with probability $\frac14$.
\end{remark}

\begin{proof}
There clearly exists a unique choice for $g(x)$ such that $\exp A^N_1(x) = g(x) U^N_1(x)g(x^+_1)^{-1}$.
With this choice for $g(x)$, observe that
\begin{equ}
U^g(x,x^+_2) = e^{A^N_1(x)} e^{A^N_2(x^+_1)} e^{-\log U^g(\partial p_1)} e^{-A^N_1(x^+_2)}\;,
\end{equ}
from which it follows by Lemma~\ref{lem:CBH_bound} that
\begin{equ}\label{eq:Ug_A_sum}
\log U^g(x,x^+_2) = A^N_1(x) + A^N_2(x^+_1) - A^N_1(x^+_2) - \log U^g(\partial p_1) + O(\delta^2)\;.
\end{equ}
Furthermore, we have
\begin{equ}
e^{-2A^N_1(x^+_2)}e^{-2A^N_2(x^-_1)} e^{2A^N_1(x^-_2)} e^{2A^N_2(x^+_1)}
= \prod_{i=1}^4 x_i\;,
\end{equ}
where $x_1 = U^g(\partial p_1)$ and $x_i = u_i U^g(\partial p_i) u_i^{-1}$ for $i=2,3,4$, where $u_i$ is a suitable product of elements of the form $U^g(\partial p_i)$ and $e^{\pm A^N_\mu(x^\pm_\nu)}$, $\mu\neq \nu$.
By Lemma~\ref{lem:CBH_bound}, we have
\begin{equ}\label{eq:A_sum_p_i}
A^N_1(x^-_2)-A^N_1(x^+_2)+A^N_2(x^+_1)-A^N_2(x^-_1) = \frac12\sum_{i=1}^4 \log U(\partial p_i) + O(\delta^2)\;.
\end{equ}
Combining~~\eqref{eq:Ug_A_sum},~\eqref{eq:A_sum_p_i}, and the definition of $A^N_1(x)$, we obtain
\begin{equ}
\log U^g(x,x^+_2) = \frac{A^N_2(x^+_1) + A^N_2(x^-_1)}{2}+ \frac{3}{8}\partial_1 F_{21}(x) + \frac{1}{8}\partial_1 F_{21}(x^-_2) + O(\delta^2)\;,
\end{equ}
from which the existence of $\mcE_1$ with the desired property follows.
The existence of $\mcE_2$ and $\mcE_3$ follows in the same manner.
\end{proof}

We now extend the definition of $A$ and $g$ to $\Bonds_N$ and $\Lambda_N$ as in Lemma~\ref{lem:mcE_bound} choosing $\mcE_i$ in an arbitrary way provided the bound $|\mcE_i| \leq C\delta^2$ is satisfied.
By induction, we define $A \in \Omega^{1,(N_1)}$ such that $\exp A = U^g$ as desired.

We now show that this choice leads to a bound on $|A|^{(N_1)}_{\bar\alpha}$.
In the following, we use the shorthand $|A^N(x)| \eqdef \max_{\mu\in\{1,2\}} |A^N_\mu(x)|$.

\begin{lemma}[Bonds bound]\label{lem:bonds_bound_non-Abelian}
Suppose there exists $\alpha \in (0,1)$ and $\bar C \geq 0$ such that~\eqref{eq:hol_bound_rectangle} holds for all plaquettes $r\subset \Lambda_N$ for all $N_0 \leq N \leq N_1$.
Then there exists $C \geq 0$, not depending on $N_1$, such that if
\begin{equ}\label{eq:initial_bound}
\bar C2^{-{N_0}\alpha} + \max_{x \in \Lambda_{N_0}} |A^{N_0}(x)| \leq c\;,
\end{equ}
where $c \in (0,\infty]$ is a constant depending only on $G$, then for all $N_0\leq N \leq N_1$
\[
\max_{x\in\Lambda_N}|A^N(x)| \leq C 2^{-N\alpha}\;.
\]
\end{lemma}

\begin{proof}
Fix any $\eps \in (0,\frac12)$ and consider $N>N_0$.
We may suppose that $\bar C2^{-N_0\alpha} \leq 1$.
Using Lemma~\ref{lem:mcE_bound} and the assumption that~\eqref{eq:hol_bound_rectangle} holds for every plaquette, we have
\begin{equ}\label{eq:first_induct}
\max_{x\in\Lambda_N} |A^N(x)| \leq \delta/2 + C_1(\bar C 2^{-N\alpha} + \delta^2)\;,
\end{equ}
where $C_1$ depends only on $G$ and
$
\delta \eqdef \max_{x\in\Lambda_{N-1}}|A^{N-1}(x)|
$.
Provided that $\delta \leq \eps/C_1$, we have
\[
\max_{x\in\Lambda_N} |A^N(x)| \leq (\eps+1/2) \delta + C_1\bar C 2^{-N\alpha}\;.
\]
If $\bar C 2^{-N\alpha}$ is furthermore sufficiently small, we have
\[
(\eps+1/2) \delta + C_1\bar C 2^{-N\alpha} \leq \eps/C_1\;.
\]
We conclude that there exists $c>0$, depending only on $G$, such that if~\eqref{eq:initial_bound} holds, then~\eqref{eq:initial_bound} also holds with $N_0$ replaced by $N > N_0$ and
\begin{equ}\label{eq:second_induc}
\max_{x\in\Lambda_N} |A^N(x)| \leq (\eps+1/2)\max_{x \in \Lambda_{N-1}} |A^{N-1}(x)| + C_2 2^{-N\alpha}\;,
\end{equ}
where $C_2$ does not depend on $N$.
Proceeding by induction and lowering $\eps$ if necessary so that $\theta \eqdef (\eps+1/2)2^{\alpha} < 1$ we see that
\begin{equs}
\max_{x\in\Lambda_N} |A^N(x)|
&\leq (\eps+1/2)^{N-N_0}\Big(\max_{x\in\Lambda_{N_0}} |A^{N_0}(x)|\Big) + C_2  2^{-N\alpha} \sum_{k=0}^{N-N_0} \theta^k
\\
&\leq C_3 2^{-N\alpha}\;,
\end{equs}
where $C_3$ can depend on $\theta$ and $N_0$ but not on $N$.
\end{proof}

\begin{lemma}\label{lem:hol_sum_bound}
Let $\bar\alpha \in (\frac12,1)$ and $q \in [1,\frac{1}{1-\bar\alpha})$.
Then for every rectangle $r \subset \Lambda_N$ it holds that
\[
\Big| \sum_{i=1}^k \log U^g(\partial p_i) \Big| \leq C \Big(1+(k2^{-N})^{\bar\alpha} |A|^{(N-1)}_{\gr{\bar\alpha}} \Big) |X|_{\var q}\;,
\]
where $X$ is the anti-development of $U$ along $r$, $p_1,\ldots,p_k$ are all the plaquettes contained in $r$, and $C$ is a constant depending only on $G$, $\bar\alpha$, and $q$.
\end{lemma}

\begin{proof}
The idea is to write $\sum_{i=1}^k \log U^g(\partial p_i)$ as a Young integral against the anti-development of $U$ along $r$.
Using the notation from Definition~\ref{def:development},
let $\ell_i$ be the unique line contained in the boundary of $r$ which connects $z$, the origin of $r$, and $z_i$, the origin of $p_i$.
Note that $\ell_i \in \mcX^{(N-1)}$.
Writing $v_i \eqdef U(\ell_i) U(\partial p_i) U(\ell_i)^{-1}$, observe that $\log U(\partial r_i) = v_1\ldots v_i$, and thus $X_j = \sum_{i=1}^j \log v_i$.
Observe further that $U^g(\partial p_i) = x_i v_i x_{i}^{-1}$ where $x_i \eqdef g(z_i) U(\ell_i)^{-1} = U^g(\ell_i)^{-1} g(z)$.
Defining the $\Aut(\mfg)$-valued sequence $Y_i = \Ad_{x_i}$, it holds that
\begin{equ}
\sum_{i=1}^k \log U^g(\partial p_i) = \sum_{i=1}^k Y_i(X_i-X_{i-1})\;,
\end{equ}
which is in the form of a Young integral.
Using that $\exp(A) = U^g$ on $\Bonds_{N-1}$, we see that $(Y_i)_{i=1}^k$ is the development of $(-A(\ell_i))_{i=1}^k$ into $\Aut(\mfg)$ (through left multiplication in the adjoint representation) with initial point $Y_1 = \Ad_{g(z)}$.
By Lemma~\ref{lem:ellA_Hol_bound}, it holds that the $\bar\alpha^{-1}$-variation of the sequence $(-A(\ell_i))_{i=1}^k$ is bounded above by $(k2^{-N})^{\bar\alpha}|A|^{(N-1)}_{\gr{\bar\alpha}}$, and thus Young's estimate for controlled ODEs implies
\begin{equ}
|Y|_{\var{\bar\alpha^{-1}}} \lesssim (k2^{-N})^{\bar\alpha}|A|^{(N-1)}_{\gr{\bar\alpha}}\;.
\end{equ}
Since $q^{-1} + \bar \alpha > 1$ and since $|Y_1|=1$ (in fact $|Y_i|=1$ for all $i=1,\ldots, k$), the conclusion follows by Young integration.
\end{proof}

\begin{theorem}[Binary Landau gauge]\label{thm:Landau}
Suppose there exists $\alpha\in(\frac23,1)$, $\bar C \geq 0$, and $q \in [1,\frac{1}{1-\alpha})$ such that~\eqref{eq:hol_bound_antidev} holds for all rectangles $r \subset \Lambda_N$ and $N_0 \leq N \leq N_1$.
Suppose further that~\eqref{eq:initial_bound} holds.
Then for every $\bar\alpha \in (1-q^{-1},\alpha)$ there exists $C \geq 0$, not depending on $N_1$, such that for all parallel $\ell,\bar\ell \in \mcX^{(N_1)}$
\begin{equ}\label{eq:Aell_bound}
|A(\ell)| \leq C|\ell|^{\bar \alpha}
\end{equ}
and
\begin{equ}\label{eq:Aell_barell_bound}
|A(\ell) -  A(\bar\ell)| \leq C |\ell|^{\alpha/2} d(\ell,\bar\ell)^{\bar\alpha/2}\;.
\end{equ}
\end{theorem}

\begin{proof}
It suffices to consider $\bar\alpha \in (\frac23 \vee (1-q^{-1}),\alpha)$.
To prove~\eqref{eq:Aell_bound}, we proceed by induction on $N \geq N_0$.
Assume that $|A(\ell)| \leq P_{N-1}|\ell|^{\bar\alpha}$ for some constant $P_{N-1} \geq 1$ and all $\ell \in \mcX^{(N-1)}$.

Let $\ell \in \mcX^{(N)}$.
Suppose first that $\ell$ is contained in $\Grid_{N-1}$, the grid of $\Lambda_{N-1}$.
Then we can write $\ell = \ell_1 \cup \ell_2 \cup \ell_3$ where $\ell_1 \in \mcX^{(N-1)}$ and, for $i=2,3$, $\ell_i$ is either empty or is a bond of $\Lambda_N$.
By induction, we know that $|A(\ell_1)| \leq P_{N-1} |\ell_1|^{\bar\alpha}$.
If both $\ell_2,\ell_3$ are empty, then we are done.
Otherwise, by Lemma~\ref{lem:bonds_bound_non-Abelian}, we have $|A(\ell_2)|+|A(\ell_3)| \leq C_1 2^{-N\alpha}$ for a constant $C_1$ not depending on $N$.
If $\ell_1$ is empty, then again we are done by choosing $P_N \geq C_1$.
Otherwise we have
\begin{equ}
(|\ell_1|+2^{-N})^{\bar\alpha} - |\ell_1|^{\bar\alpha} = \int_0^{2^{-N}} \bar\alpha(|\ell_1|+r)^{\bar\alpha-1} \mrd r \geq \bar\alpha 2^{-N\bar\alpha}\;.
\end{equ}
Since $C_1$ is independent of $N$, we may increase $P_{N-1}$ if necessary so that $P_{N-1} \bar\alpha \geq C_1$.
Hence
\begin{equs}
|A(\ell)|
&= |A(\ell_1) + A(\ell_2) + A(\ell_3)|
\\
&\leq P_{N-1} |\ell_1|^{\bar\alpha} + C_1 2^{-N\bar\alpha}
\leq  P_{N-1} (|\ell_1|+ 2^{-N})^{\bar\alpha} \leq  P_{N-1}|\ell|^{\bar\alpha}\;,
\end{equs}
which proves the inductive step in the case $\ell \subset \Grid_{N-1}$.
Note that the same constant $P_{N-1}$ appears, which will be used in the next case.

Suppose now $\ell$ is not contained in $\Grid_{N-1}$.
Then by the definition of $A^N$, we have
\begin{equ}\label{eq:ell_average}
A(\ell) = \frac{A(\ell_1) + A(\ell_2)}{2} + \Delta_1 + \Delta_2\;.
\end{equ}
where $\ell_1,\ell_2 \in\mcX^{(N)}$ are parallel to $\ell$ and are contained in $\Grid_{N-1}$.
Here $\Delta_1$ accounts for the terms $\partial_\mu F_{\nu\mu}$ and satisfies for a constant $C_2$ depending only on $G$, $q$, and $\bar\alpha$
\begin{equs}
|\Delta_1|
&\leq \Big| \sum_{p} \log U^g(\partial p) \Big|
\\
&\leq C_2 (1+ |\ell|^{\bar\alpha}P_{N-1}) \bar C |\ell|^{\alpha/2} 2^{-N\alpha/2}\label{eq:Delta1_bound}
\\
&\leq 2\cdot C_2 P_{N-1}\bar C |\ell|^{\bar\alpha}2^{-N(\alpha-\bar\alpha)} \;,
\end{equs}
where the sum is taken over all plaquettes $p\subset \Lambda_N$ which have a corner belonging to $\ell$ and the second inequality is due to Lemma~\ref{lem:hol_sum_bound}.

The term $\Delta_2$ accounts for the errors $\mcE_i$ from the CBH formula and satisfies, by Lemma~\ref{lem:mcE_bound}, for a constant $C_3$ depending only on $G$
\begin{equ}\label{eq:Delta2_bound}
|\Delta_2| \leq C_3 |\ell| 2^{N} (\bar C + C_1)^2 2^{-2N\alpha} \leq C_4 |\ell| 2^{-N(2\alpha-1)}\;,
\end{equ}
where we have used that~\eqref{eq:hol_bound_antidev} holds for all plaquettes, Lemma~\ref{lem:bonds_bound_non-Abelian} as above, and the fact that $\ell$ is a union of $|\ell|2^N$ bonds of $\Lambda_N$.

Using these estimates for $\Delta_1,\Delta_2$, it follows from the previous case that
\begin{equ}
|A(\ell)| \leq P_{N-1}|\ell|^{\bar\alpha} + C_5 2^{-N(\alpha-\bar\alpha)} P_{N-1} |\ell|^{\bar\alpha}
\end{equ}
for $C_5$ independent of $N$.
Hence we have shown the inductive step with $P_N \eqdef P_{N-1}(1 + C_5 2^{-N(\alpha-\bar\alpha)})$, and thus $\sup_N P_N < \infty$.
This completes the proof of~\eqref{eq:Aell_bound}.

To prove~\eqref{eq:Aell_barell_bound}, we again proceed by induction on $N$.
Suppose that the case $N-1$ holds with proportionality constant $Q_{N-1}$.
Let $\ell,\bar\ell \in \mcX^{(N)}$ be distinct and parallel.
Suppose first that $\ell$ and $\bar\ell$ are both contained in $\Grid_{N-1}$.
We write $\ell=\ell_1\cup\ell_2\cup \ell_3$ as before and similarly for $\bar\ell$.
Note that we can take parallel $\ell_1,\bar\ell_1 \in \mcX^{(N-1)}$ to which we can apply the inductive hypothesis.
If $\ell_2$ and $\ell_3$ are both empty, or if $\ell_1$ is empty, then we are done.
Otherwise, in the same way as the proof of~\eqref{eq:Aell_bound},
\begin{equs}
|A(\ell) - A(\bar\ell)|
&\leq |A(\ell_1)-A(\bar\ell_1)| + \sum_{i=2}^3 |A(\ell_i)| + |A(\bar\ell_i)|
\\
&\leq Q_{N-1}|\ell_1|^{\alpha/2} d(\ell,\bar\ell)^{\bar\alpha/2} + C_6 2^{-N\alpha}
\\
&\leq Q_{N-1}(|\ell_1|+2^{-N})^{\alpha/2}d(\ell,\bar\ell)^{\bar\alpha/2}
\\
&\leq Q_{N-1}|\ell|^{\alpha/2}d(\ell,\bar\ell)^{\bar\alpha/2}
\end{equs}
(where we increase $Q_{N-1}$ if necessary as before).
Now suppose $\bar\ell$ is contained in $\Grid_{N-1}$ and $\ell$ is not.
Then we know $A(\ell)$ admits the expression~\eqref{eq:ell_average} with the same bounds on $\Delta_1$ and $\Delta_2$, and where $\ell_1$ and $\ell_2$ are parallel to $\bar\ell$ with
\[
d(\ell_1,\bar\ell) = d(\ell,\bar\ell) - 2^{-N} \quad \textnormal{ and } \quad d(\ell_2,\bar\ell) \leq d(\ell,\bar\ell) + 2^{-N}\;.
\]
By the previous case and the concavity of $x\mapsto x^{\bar\alpha/2}$, we have
\begin{equ}
|A(\ell)-A(\bar\ell)|
\leq Q_{N-1} |\ell|^{\alpha/2}d(\ell,\bar\ell)^{\bar\alpha/2} + \Delta_1 + \Delta_2\;.
\end{equ}
From~\eqref{eq:Delta1_bound} we have
\[
\Delta_1 \leq C_7 |\ell|^{\alpha/2}d(\ell,\bar\ell)^{\bar\alpha/2} 2^{-N(\alpha-\bar\alpha)/2}
\]
($C_7$ takes into account the fact that $\sup_N P_{N} < \infty$).
From~\eqref{eq:Delta2_bound} and the condition $\frac23 < \bar\alpha < \alpha$ we have
\[
\Delta_2 \leq C_4|\ell| d(\ell,\bar\ell)^{\bar\alpha/2}2^{-N(\alpha-\bar\alpha)}\;.
\]
It follows that
\begin{equ}\label{eq:ell_barell_bound}
|A(\ell)-A(\bar\ell)| \leq (Q_{N-1} + C_8 2^{-N(\alpha-\bar\alpha)/2}) |\ell|^{\alpha/2}d(\ell,\bar\ell)^{\bar\alpha/2}
\end{equ}
for $C_8$ independent of $N$.
For the final case, when neither $\ell$ nor $\bar\ell$ are contained in $\Grid_{N-1}$, we write $A(\ell)$ and $A(\bar\ell)$ as in~\eqref{eq:ell_average} with corresponding $\Delta_i, \bar\Delta_i$ and parallel $\ell_i, \bar\ell_i$ which are contained in $\Grid_{N-1}$ and $d(\ell_i,\bar\ell_i) = d(\ell,\bar\ell)$ for $i=1,2$.
By exactly the same argument we again obtain~\eqref{eq:ell_barell_bound}.
Hence we have shown the inductive step with $Q_N \eqdef Q_{N-1}+C_8 2^{-N(\alpha-\bar\alpha)/2}$, and thus $\sup_N Q_N < \infty$, which completes the proof of~\eqref{eq:Aell_barell_bound}.
\end{proof}

\subsection{Axial gauge}\label{subsec:axial_gauge}

In this subsection we conclude the proof of  Theorem~\ref{thm:Landau_axial} by showing that an axial-type gauge gives an easy bound of the order $|A^N_\mu(x)| \lesssim 2^{-N\alpha/2}$, which ensures we can always start the induction in Lemma~\ref{lem:bonds_bound_non-Abelian}.

\begin{remark}
This is the only part where we use simple connectedness of $G$.
If we chose to work on $[0,1]^2$ instead of $\T^2$, then this assumption could be dropped and a simplified version of the gauge presented in this subsection could be used. 
\end{remark}

Consider $N \leq N_1$ and treat in this subsection $U$ only as a function in $\mfA^{(N)}$.
We define a gauge transform $g\in\mfG^{(N)}$ as follows.
For $n=0,\ldots, 2^N-1$, let $y_n \eqdef (0,n 2^{-N}) \in \Lambda_N$ and define
\begin{equ}
V \eqdef U(y_0,y_1)\ldots U(y_{2^N-1},y_0)\;.
\end{equ}
There exists a unique $\bar g \in \mfG^{(N)}$ with support on $\{y_1,\ldots,y_{2^N-1}\}$ such that $U^{\bar g}(y_n,y_{n+1}) = \exp(2^{-N}\log V)$.
Define further
\begin{equ}
U_n \eqdef U^{\bar g}(y_n,y_n+2^{-N} e_1)\ldots U^{\bar g}(y_n+(2^{N}-1)2^{-N} e_1,y_n)\;.
\end{equ}
Note that for any $u^n_0, \ldots, u^n_{2^N-1} \in G$ for which $u^n_0\ldots u^n_{2^N-1}=U_n$, there exists a unique $g \in \mfG^{(N)}$ such that $g(0) = 1_G$, $g \equiv \bar g$ on $\{y_1,\ldots,y_{2^N-1}\}$, and
\begin{equ}\label{eq:def of tilde U}
U^g(y_n+m2^{-N} e_1,y_n+(m+1)2^{-N} e_1) = u^n_m
\end{equ}
for all $m,n\in\{0,\ldots, 2^N-1\}$.
We require the following lemma from quantitative homotopy theory.

\begin{lemma}\label{lem:quant_homotopy}
Suppose $G$ is simply connected.
There exist $\delta_0, C > 0$, depending only on $G$, with the following property.
Let $u_0,\ldots,u_k\in G$ such that $|\log(u_n^{-1}u_{n+1})| \leq \delta \leq \delta_0$ for all $n\in\{0,\ldots, k-1\}$.
Then there exist $\gamma_0,\ldots, \gamma_k \in C([0,1],G)$ such that $\gamma_n(0) = 1_G$, $\gamma_n(1)=u_n$, and for all $n\in\{0,\ldots, k-1\}$ and $s,t \in [0,1]$
\begin{equ}\label{eq:gamma_n_bound}
|\log(\gamma_n(t)^{-1}\gamma_{n+1}(t))| \leq C(\delta + k^{-1})\;,
\end{equ}
and
\begin{equ}\label{eq:gamma_st_bound}
|\log(\gamma_n(s)^{-1}\gamma_n(t))| \leq C (k \delta + 1) |t-s|\;.
\end{equ}
\end{lemma}

\begin{proof}
Consider the path $\gamma:[0,1]\to G$ for which $\gamma(0) = u_0$ and $\gamma$ restricted to $[n/k,(n+1)/k]$ is the geodesic (one-parameter subgroup) from $u_n$ to $u_{(n+1)/k}$.
Then $\gamma$ is $L$-Lipschitz with $L\eqdef k\delta$ (for the geodesic distance on $G$).
Since $G$ is simply connected, it is well-known that there exists a $C(L + 1)$-Lipschitz homotopy $H : [0,1]\times [0,1] \to G$ taking $\gamma$ to the constant path at $1_G$, see e.g.~\cite[Theorem~B]{Chambersetal16}.
Setting $\gamma_n(t) \eqdef H(n/k,t)$ concludes the proof.
\end{proof}

\begin{proposition}[Axial gauge]\label{prop:axial}
Suppose $G$ is simply connected.
Then for every $\bar C \geq 0$ and $\alpha \in (0,2)$, there exists $N_0 \geq 0$ and $C \geq 0$ such that for all $N \geq N_0$, if
\begin{equ}\label{eq:rectangle_assump}
\textnormal{\eqref{eq:hol_bound_rectangle} holds for all rectangles $r=(y_n, m2^{-N} e_1, 2^{-N}e_2)$}\;,
\end{equ}
where $y_n = (0,n2^{-N}) \in\Lambda_N$, $0 \leq n < 2^N$, and $1 \leq m < 2^N$,
then there exists $g \in \mfG^{(N)}$ such that $A \eqdef \log U^g$ satisfies for $\mu=1,2$
\begin{equ}\label{eq:max_A_bound}
\max_{x \in \Lambda_{N}} |A^N_\mu(x)| \leq C 2^{-N\alpha/2}\;.
\end{equ}
\end{proposition}

\begin{proof}
Defining $\bar g$ as above, observe that~\eqref{eq:rectangle_assump}, together with the fact that $\alpha<2$ and $|\log U^{\bar g}(y_n,y_{n+1})| \lesssim 2^{-N}$, implies that $|\log U_n^{-1}U_{n+1}| \lesssim 2^{-N\alpha/2}$.
We are thus able to apply Lemma~\ref{lem:quant_homotopy} with $k=2^N$, $u_n = U_n$ for $n=0,\ldots,2^{N}-1$ and $u_{2^N} = U_0$,
and $\delta \lesssim 2^{-N\alpha/2}$.
We then define $u^n_m \eqdef \gamma_n(m2^{-N})^{-1}\gamma_n((m+1)2^{-N})$ for $m=0,\ldots, 2^N-1$, and choose the unique corresponding $g\in\mfG^{(N)}$ as dictated above.
The bound~\eqref{eq:max_A_bound} follows for $\mu=1$ from the definition of $u^n_m$ and~\eqref{eq:gamma_st_bound}, and for $\mu=2$ from~\eqref{eq:gamma_n_bound} and~\eqref{eq:rectangle_assump}.
\end{proof}

\section{Probabilistic bounds}
\label{sec:prob_bounds}

In this section we show that discrete approximations of the Yang--Mills measure satisfy the bounds required in Theorem~\ref{thm:Landau_axial}.

For every $N \geq 0$, let $Q_N : G \to [0,\infty)$ be measurable map such that $\int_G Q_N(x) \mrd x = 1$, and $Q_N(x)=Q_N(x^{-1})$ and $Q_N(yxy^{-1})=Q_N(x)$ for all $x,y\in G$.
Consider the probability measure on $\mfA^{(N)}$
\[
\mu_N(f) = Z_N^{-1}\int_{\mfA^{(N)}} f(U) \prod_{p \subset \Lambda_N} Q_{N}[U(\partial p)] \mrd U\;,
\]
where the product is over all plaquettes $p \subset \Lambda_N$, $\mrd U$ is the Haar measure on $\mfA^{(N)} \cong G^{|\overline\Bonds_N|}$, and $Z_N$ is the normalisation constant which makes $\mu_N$ a probability measure.

For an integer $N\geq 0$ and constants $C_l,C_u,\bar C\geq 0$ consider the conditions
\begin{equ}\label{eq:cond_mixing}
\forall x \in G\;, \quad C_l^{-1} \leq Q_N^{\star M}(x) \leq C_u\;,
\end{equ}
where $M = 1\vee 2^{2N-3}$ and $Q_N^{\star k}$ denotes the $k$-fold convolution of $Q_N$ with itself,
and for some $\beta \geq 1$
\begin{equ}\label{eq:cond_beta}
\int_G |\log x|^\beta Q_N(x) \mrd x \leq \bar C2^{-\beta N}\;.
\end{equ}
Condition~\eqref{eq:cond_mixing} means that the $G$-valued random walk with increments $Q_N(x)\mrd x$ has a density after $M$ steps which is bounded above and below.
Condition~\eqref{eq:cond_beta} means that the $\beta$-th moment of $Q_N(x)\mrd x$ is comparable to the $\beta$-th moment of $B(2^{-2N})$, where $B$ is a $G$-valued Brownian motion.

\begin{remark}
The symmetry assumption $Q_N(x)=Q_N(x^{-1})$ simplifies several points, namely the proof of Lemma~\ref{lem:BDG} below, but is not at all necessary provided we make an assumption of the type $|\int_G \log( x )Q_N(x) \mrd x | \lesssim 2^{-2N}$ to control the drift of the associated $G$-valued random walk.
\end{remark}

\begin{example}\label{ex:actions}
Two common choices for $Q_N$ are the
\begin{itemize}
\item Villain (heat kernel) action $Q_N = e^{t\Delta}$ at time $t=2^{-2N}$, where $\Delta$ is the Laplace--Beltrami operator on $G$,
\item Wilson action $Q_N(x) = \bar Z^{-1}_N\exp(\eps^{d-4}\Re\Trace(x-I))$, where $\eps=2^{-N}$, $d=2$, and we implicitly assume $G$ is a matrix Lie group.
\end{itemize}
One can check that for every $\beta \geq 1$ there exist $C_l,C_u,\bar C\geq 0$ such that~\eqref{eq:cond_mixing} and~\eqref{eq:cond_beta} hold for all $N \geq 0$ and these two choices of $Q_N$.
\end{example}

The main result of this section is the following Kolmogorov-type criterion.
We henceforth fix $N \geq 0$ and let $U$ denote the $\mfA^{(N)}$-valued random variable distributed by $\mu_{N}$.

\begin{theorem}\label{thm:hol_bound}
Let $\beta \geq 2$ and suppose that~\eqref{eq:cond_mixing} and~\eqref{eq:cond_beta} hold.
Then for any $q > 2$ and $\alpha < 1- \frac{6}{\beta}$, there exists $\lambda\geq 0$ depending only on $G,\beta,q$, such that
\[
\E \Big[ \sup_{0\leq n \leq N} \sup_{r \subset \Lambda_n} \frac{|\log U(\partial r)|^\beta + |X|_{\var q}^\beta}{|r|^{\beta\alpha/2}} \Big] \leq \lambda C_lC_u\bar C\;,
\]
where the second supremum is taken over all rectangles $r \subset \Lambda_n$, and $X$ denotes the anti-development of $U$ along $r$.
\end{theorem}

The idea of the proof is to approximate the holonomy $U(\partial r)$ and the anti-development $X$ by pinned random walks, and the latter we control using rough paths theory.
We require the following lemma.

\begin{lemma}\label{lem:BDG}
Suppose~\eqref{eq:cond_beta} holds for some $\beta \geq 2$ and $\bar C \geq 0$.
Then for all $q > 2$, there exists $\lambda \geq 1$, depending only on $G$, $\beta$ and $q$, such that for all $M, k \geq 1$
\[
\int_{G^k} \Big( |X|^\beta_{\var q} + |\log (v_1\ldots v_k)|^\beta \Big)
Q_N^{\star M}(v_1)\ldots Q_N^{\star M}(v_{k}) \mrd v \leq \lambda \bar C (k M 2^{-2N})^{\beta/2}\;,
\]
where $(X_j)_{j=0}^k$ is the sequence $X_j = \sum_{i=1}^j \log(v_i)$.
\end{lemma}

\begin{proof}
We first prove the claim for $M=1$.
Let $k \geq 1$ and consider i.i.d. $\mfg$-valued random variables $V_1, V_2, \ldots, V_k$ equal in law to $\log(Y)$, where $Y \sim Q_N(x) \mrd x$.
Consider the martingale $(X_j)_{j=0}^k$ defined by $X_j \eqdef \sum_{i=1}^j V_i$ and let $\mbX$ denote its canonical (Marcus) level-2 rough path lift (see~\cite[Sec.~4]{CF17}).
Then
\begin{equs}
\E[\|\mbX\|_{\var{q}}^\beta] &\leq C_1 \E\Big[\Big(\sum_{i=1}^k |V_i|^2 \Big)^{\beta/2}\Big]
\\
&\leq C_1 \E\Big[k^{\beta/2-1} \sum_{i=1}^k |V_i|^\beta\Big]
\\
&\leq C_1 \bar C k^{\beta/2}2^{-N\beta}\;,
\end{equs}
where $C_1$ depends only on $\beta,q$, and
where we used the enhanced BDG inequality~\cite[Thm.~4.7]{CF17} in the first inequality, the power-mean inequality in the second inequality, and~\eqref{eq:cond_beta} in the final inequality.

Note that trivially $|X|_{\var q} \leq \|\mbX\|_{\var q}$.
Note also that $e^{V_1}\ldots e^{V_k}$ is the solution to a controlled (Marcus) differential equation driven by $X$.
By the local-Lipschitz continuity of the rough path solution map, it follows that $|\log(e^{V_1}\ldots e^{V_k})| \leq C_2 \|\mbX\|_{\var q}$, where $C_2$ depends only on $G$ and $q$. This proves the claim for $M=1$.

For general $M \geq 1$, observe that taking $k=M$ in the previous case implies that~\eqref{eq:cond_beta} holds with $Q_N$ on the LHS replaced by $Q^{\star M}_N$ and $\bar C$ and $2^{-2N}$ on the RHS replaced by $\lambda \bar C$ and $M 2^{-2N}$ respectively ($\lambda$ depending only on $G,\beta,q$).
The conclusion again follows from the previous part by replacing $Q_N$ by $Q_N^{\star M}$.
\end{proof}

\begin{proof}[of Theorem~\ref{thm:hol_bound}]
Let $n \leq N$ and consider a rectangle $r \subset \Lambda_{n}$.
We first show that
\begin{equ}\label{eq:single_rectangle}
\E\Big[|\log U(\partial r)|^\beta + |X|_{\var q}^\beta \Big] \leq \lambda C_lC_u\bar C |r|^{\beta/2} \;,
\end{equ}
where $\lambda$ depends only on $G,q,\beta$.
It suffices to consider $2 \leq n \leq N$ and $r=(0,k2^{-n}e_1,2^{-n}e_2)$ where $k < 2^{n-1}$.
Note that the discrete measure $\mu_{N}$ has a domain Markov property:
if $D$ is a simply connected domain of $\Lambda_{N}$, then, conditioned on the bonds of the boundary, the measure inside $D$ is independent from the measure outside $D$.
As a consequence, we can substitute the lattice $\Lambda_N$ by the square $D = [0,\frac12]^2\cap\Lambda_N$ (which contains $r$ by assumption) with prescribed bond variables on the boundary.
More precisely, since $U(\partial r)$ and $|X|_{\var q}$ are functions only of the bond variables inside and on the boundary of $D$, we can write the LHS of~\eqref{eq:single_rectangle} as
\[
Z_N^{-1} \int_{G^{\bar K}} \mrd \bar U \Big(\int_{G^{\check K}} \mrd \check U \prod_{\check p \not\subset D} Q_N[U(\partial \check p)] \Big) \Big( \int_{G^{\mathring K}} \mrd \mathring U F(U) \prod_{\mathring p\subset D} Q_N[U(\partial \mathring p)]\Big)\;,
\]
where $F(U)=|\log U(\partial r)|^\beta + |X|_{\var q}^\beta$, and $\bar U$ are the bond variables on the boundary of $D$, $\check U$ are the bond variables outside $D$, and $\mathring U$ are the bond variables inside $D$.
Denoting
\begin{equ}
\mathring Z(\bar U) \eqdef \int_{G^{\mathring K}} \mrd \mathring U \prod_{\mathring p\subset D} Q_N[U(\partial \mathring p)]\;,
\end{equ}
we aim to show that for all $\bar U$
\begin{equ}\label{eq:barU_integral}
\int_{G^{\mathring K}} \mrd \mathring U F(U) \prod_{\mathring p\subset D} Q_N[U(\partial \mathring p)] \leq \lambda C_lC_u\bar C \mathring Z(\bar U) |r|^{\beta/2}
\end{equ}
from which~\eqref{eq:single_rectangle} follows by the definition of $Z_N$.

Suppose first that $n<N$.
To facilitate analysis of the integrals, we fix a maximal tree $\mfT \subset \overline\Bonds_N$ inside $D$ as follows.
We include in $\mfT$ all bonds on the boundary of $D$ except $\bar \alpha \eqdef ((\frac12,\frac12-2^{-N}),(\frac12,\frac12))$.
We further include all horizontal bonds $2^{-N}((x,y),(x+1,y))$ where either
\begin{itemize}
\item $x \in \{0,\ldots, 2^{N-1}-2\}$ and $y=2^{N-n}+2m$ for some integer $m \geq 0$ such that $y \in \{2^{N-n},\ldots, 2^{N-1}-1\}$, or
\item $x \in \{1,\ldots, 2^{N-1}-1\}$ and $y=2^{N-n}+(2m+1)$ for some integer $m \geq 0$ such that $y \in \{2^{N-n},\ldots, 2^{N-1}-1\}$
\end{itemize}
and all vertical bonds $2^{-N}((x,y),(x,y+1))$ where either
\begin{itemize}
\item $y \in \{1,\ldots 2^{N-n}-1\}$ and $x$ is odd and $x \in \{1,\ldots, 2^{N-1}-1\}$, or
\item $y \in \{0,\ldots 2^{N-n}-2\}$ and $x$ is even and $x \in \{1,\ldots, 2^{N-1}-1\}$.
\end{itemize}
See Figure~\ref{fig:tree} for an example of $\mfT$.

\begin{figure}[t]
\centering
\begin{tikzpicture}[scale = 1.6]
\foreach \x in {0,...,6}{
  \foreach \y in {0,...,6}{
    \fill[black] (\x,\y) circle (0.05);
  }
}
\draw[thick] (5,1)--++(-4,0)--++(0,4)--++(4,0);
\draw[thick] (5,1)--++(0,3+3/4);

\draw[thick,densely dotted] (5,5)--++(0,-1/4);
\node (af) at (5+4/32,5-1/8) {\small{$\bar \alpha$}};

\draw[thick,densely dotted] (1,2-1/4)--++(1/4,0);
\node (a1) at (1+9/64,2-1/4-1/8) {\small{$\alpha_1$}};

\draw[thick,densely dotted] (5-1/4,5)--++(0,-1/4);
\node (aN) at (5-1/4-5/32,5-1/8) {\small{$\alpha_K$}};

\foreach \x in {0,...,24}{
  \foreach \y in {0,...,24}{
    \fill[black] (0+\x/4,0+\y/4) circle (0.02);
  }
}
\foreach \y in {0,...,5}{
\draw[thick] (1,2+\y/2)--++(3+3/4,0);
}
\foreach \y in {0,...,5}{
\draw[thick] (1+1/4,2+1/4+\y/2)--++(3+3/4,0);
}
\foreach \x in {0,...,6}{
\draw[thick] (1+1/2+\x/2,1)--++(0,3/4);
}
\foreach \x in {0,...,7}{
\draw[thick] (1+1/4+\x/2,2)--++(0,-3/4);
}
\node (0) at (1-1/8,1) {\small{$0$}};
\end{tikzpicture}
\caption{Example of $\mfT$ for $N=5$, $n=3$.
Large circles indicate points of $\Lambda_n$, small circles indicate points of $\Lambda_N$.
Solid lines constitute bonds in $\mfT$.}\label{fig:tree}
\end{figure}
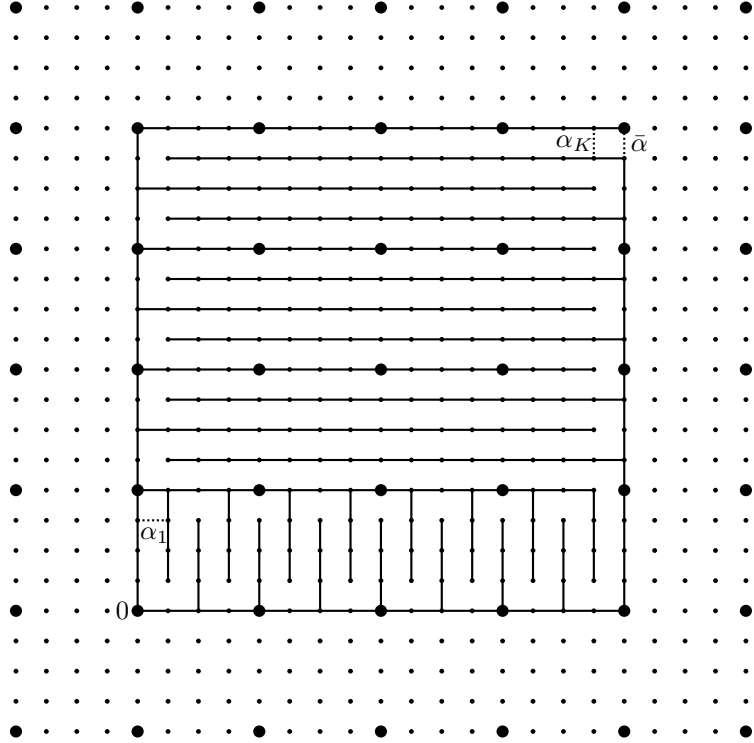

Since the integrand in~\eqref{eq:barU_integral} is gauge invariant, we can assume that $U(\alpha)=1_G$ for all $\alpha \in \mfT$ and $U(\bar\alpha) = U(\partial D)$, and take the integral over the remaining $K \eqdef 2^{2N-2}-1$ bonds in $D$.
Let us order these bonds $\alpha_1,\ldots, \alpha_{K}\in\overline\Bonds_N$ so that $\alpha_{K} = \bar\alpha - (2^{-N},0)$ with earlier bonds moving along the path traced out by $\mfT$ (see Figure~\ref{fig:tree}).
Using the shorthand $M \eqdef 2^{2(N-n)}$, $u_j\eqdef U(\alpha_j)$, and writing $r_1,\ldots, r_k \subset \Lambda_n$ for the subrectangles of $r$ as in Definition~\ref{def:development}, it follows that $U(\partial r_j) = u_{jM}$.
We can then rewrite the LHS of~\eqref{eq:barU_integral} as
\begin{equ}
\int_{G^K} \mrd u \big(|X|_{\var q}^\beta + |\log u_{k M}|^\beta\big) \prod_{i=1}^{K+1} Q_N(w_i)\;,
\end{equ}
where $(X_j)_{j=0}^k$ is the sequence $X_0=0$, $X_j =\log(u_{(j-1)M}^{-1}u_{jM})$, and where $w_1 = u_1$, $w_{K+1} = u_K^{-1}U(\bar\alpha)$, and every $w_i$, $i=2,\ldots,K$, is of the form $w_i = u_{i-1}^{\pm 1} u_{i}^{\pm 1}$ (not necessarily the same choice of $\pm$ in the exponents) such that $w_{i-1}$ and $w_i$ carry opposite exponents for $u_{i-1}$ for all $i=2,\ldots, K+1$.
In Figure~\ref{fig:tree}, for example, one has 
\begin{equs}
w_2&=u_1^{-1}u_2\;,\quad
w_3=u_2^{-1}u_3\;, \quad
w_{4}=u_3^{-1}u_4\;, \quad
w_5=u_4^{-1}u_5^{-1}\;,
\\
w_6 &= u_5u_6^{-1}\;, \quad \ldots\;, \quad
w_K = u_{K-1}^{-1}u_K\;.
\end{equs}
(Note that $w_1,\ldots,w_{K+1}$ are the increments of the pinned random walk from $1_G$ to $U(\bar\alpha)$ alluded to in the remark following Theorem~\ref{thm:hol_bound}.)
In particular, we have $u_{jM} = w_1\ldots w_{jM}$ for all $j < 2^{n-1}$.
Therefore, making the change of variable $v_j = w_{(j-1)M+1}\ldots w_{jM}$, we can rewrite the LHS of~\eqref{eq:barU_integral} as
\begin{equ}\label{eq:pinned_walk}
\int_{G^{k}} \mrd v F(v) Q_N^{\star M}(v_1)\ldots Q_N^{\star M}(v_k) Q_N^{\star(K+1-kM)}(v_k^{-1}\ldots v_1^{-1} U(\bar\alpha))\;,
\end{equ}
where now $F(v) = |\log(v_1\ldots v_k)|^\beta + |X|_{\var q}^\beta$ and $X_j = \sum_{i=1}^j\log(v_i)$ for $j=0,\ldots, k$.
Note that $K+1-kM \geq 2^{2N-3}$, and thus Lemma~\ref{lem:BDG} and the upper bound in~\eqref{eq:cond_mixing} imply that~\eqref{eq:pinned_walk} is bounded above by $\lambda C_u\bar C(kM2^{-2N})^{\beta/2}$ for any value of $U(\bar\alpha)$.
Finally, we clearly have $\mathring Z(\bar U) = Q_N^{\star (K+1)}(U(\bar\alpha))$, hence using the lower bound in~\eqref{eq:cond_mixing} concludes the proof of~\eqref{eq:single_rectangle} for $n<N$. 

The case $n=N$ follows by similar (even simpler) considerations.
The only changes which need to be made are that $\mfT$ has no vertical bonds which are not on the boundary of $D$, $\bar\alpha$ is now on the north-west corner, i.e., $\bar\alpha \eqdef ((0,\frac12 - 2^{-N}),(0,\frac12))$, and correspondingly $\alpha_K  = \bar\alpha+(2^{-N},0)$, etc. $\alpha_1 = ((2^{-N},0),(2^{-N},2^{-N}))$.
Furthermore $M=1$ and the variables $(w_i)_{i=1}^{K+1}$ now satisfy $w_{K+1}=u_{K}U(\bar\alpha)^{-1}$ along with the previous conditions.
The same argument with $U(\bar\alpha)$ replaced by $U(\bar\alpha)^{-1}$ proves~\eqref{eq:single_rectangle}.

We now have
\begin{equs}
\E\Big[ \sup_{0 \leq n \leq N} \sup_{r \subset \Lambda_n} \Big|\frac{|\log U(\partial r)|^\beta + |X|_{\var q}^\beta}{|r|^{\beta\alpha/2}} \Big| \Big]
&\lesssim \sum_{n = 0}^N \sum_{ r\subset \Lambda_n} |r|^{-\beta\alpha/2+\beta/2}
\\
&\leq \sum_{n =0 }^N \sum_{x \in \Lambda_n} \sum_{k=1}^{2^n}(k2^{-2n})^{\beta(1-\alpha)/2}
\\
&\leq \sum_{n=0}^N 2^{2n}2^{n}2^{n\beta(1-\alpha)/2}2^{-n\beta(1-\alpha)}\;.
\end{equs}
The final term is bounded above independently of $N$ provided $3-\beta(1-\alpha)/2 < 0$, i.e., $\alpha < 1- 6/\beta$.
\end{proof}

\begin{proof}[of Theorem~\ref{thm:main_thm}]
Applying Theorem~\ref{thm:hol_bound} to the heat kernel action from Example~\ref{ex:actions}, Theorem~\ref{thm:Landau_axial} shows that for every $N \geq 1$, there exist an $\Omega^{1,(N)}(\T^2,\mfg)$-valued random variable $A^{(N)}$ for which $(|A^{(N)}|^{(N)}_{\alpha})_{N \geq 1}$ is tight for any $\alpha \in (0,1)$, and such that the associated gauge field induces the discrete YM measure on the lattice $\Lambda_N$.
Recall that, by Young integration, the development map $\mcC^{\Hol\alpha}([0,1],\mfg) \to \mcC^{\Hol\alpha}([0,1],G)$ is continuous (locally Lipschitz) for all $\alpha \in (\frac12,1]$. We thus obtain for any $\alpha\in(\frac12,1)$ the existence of an $\Omega_\alpha$-valued random variable $A$ with the desired properties from Lemma~\ref{lem:ellA_Hol_bound}, Theorem~\ref{thm:proj_limit}, and the characterisation of the YM measure in~\cite[Thm.~2.9.1]{Levy03}.
The fact that $A$ has support in $\Omega^1_\alpha$ follows from Proposition~\ref{prop:omega_omega1_inclusion}.
\end{proof}

\appendix

\section{Symbolic index}

We collect in this appendix commonly used symbols of the article, together
with their meaning and, if relevant, the page where they first occur.

 \begin{center}
\renewcommand{\arraystretch}{1.1}
\begin{longtable}{lll}
\toprule
Symbol & Meaning & Page\\
\midrule
\endfirsthead
\toprule
Symbol & Meaning & Page\\
\midrule
\endhead
\bottomrule
\endfoot
\bottomrule
\endlastfoot
 $A$ & $1$-form & \\
 $\mfA^{(N)}$ & Space of discrete gauge fields $U : \Bonds_N \to G$ & \pageref{page ref mfA} \\
 $\Bonds_N$ & Oriented bonds of the lattice $\Lambda_N$ & \pageref{page ref Bonds}\\
 $\overline\Bonds_N$ & Subset of $\Bonds_N$ consisting of positively oriented bonds & \pageref{page ref overline Bonds}\\
 $d_\Haus$ & Hausdorff metric on $\mcX$ & \pageref{page ref dH}\\
 $G$ & Compact, connected Lie group & \pageref{page ref G}\\
 $\mfg$ & Lie algebra of $G$ & \pageref{page ref mfg}\\
 $\mfG^{(N)}$ & Space of discrete gauge transforms $g : \Lambda_N \to G$ & \pageref{page ref mfG}\\
 $\Grid_N$ & Grid associated with $\Lambda_N$ & \pageref{page ref Grid}\\
 $\Lambda_N$ & Lattice of $\T^d$ with spacing $2^{-N}$ & \pageref{page ref Lambda}\\
 $\log$ & Right inverse of $\exp : \mfg \to G$ & \pageref{page ref log}\\
 $\ell_A$ & Path constructed from $A\in\Omega$ and $\ell \in\mcX$ & \pageref{page ref ellA}\\
 $\Omega$ & Additive functions on $\mcX$ & \pageref{Omega page ref}\\
 $\Omega^1$ & Space of bounded, measurable $1$-forms & \pageref{Omega1 page ref}\\
 $\Omega_\alpha$ & Banach space of $A \in \Omega$ with $|A|_\alpha<\infty$ & \pageref{page ref Omega alpha}\\
 $\Omega^1_\alpha$ & Closure of $\{A \in \Omega^1 \ssep |A|_\alpha < \infty\}$ in $\Omega_\alpha$ & \pageref{page ref Omega 1 alpha}\\
 $\Omega^{1,(N)}$ & Space of discrete $1$-forms $A : \Bonds_N \to E$ & \pageref{page ref Omega1N}\\
 $p$ & Plaquette of $\Lambda_N$ & \pageref{page ref plaquette}\\
 $r$ & Rectangle of $\Lambda_N$ & \pageref{page ref rectangle}\\
 $\rho(\ell,\bar\ell)$ & Area distance between parallel $\ell,\bar\ell \in \mcX$ & \pageref{rho page ref}\\
 $U$ & Discrete gauge field $U \in \mfA^{(N)}$ & \\
 $U^g$ & Gauge transform $U^g(x,y) = g(x)U(x,y)g(y)^{-1}$ & \pageref{page ref gauge transform} \\
 $U(\partial r)$ & Holonomy of $U$ around $r$ & \pageref{page ref U partial r} \\
 $\mcX$ & Set of line segments $\ell = \{x+c e_\mu \ssep c \in [0,\lambda]\}$ & \pageref{page ref mcX}\\
 $\mcX^{(N)}$ & Subset of $\ell \in \mcX$ consisting of unions of bonds in $\overline\Bonds_N$ & \pageref{page ref mcX}
 \end{longtable}
 \end{center}

\endappendix
\bibliographystyle{./Martin}
\bibliography{./refs}

\end{document}